\documentclass[12pt,a4paper]{article}
\usepackage{amsmath,amssymb,amsthm,graphicx,color,bbm}
\usepackage[left=1in,right=1in,top=1in,bottom=1in]{geometry}
\usepackage{cite}
\usepackage[british]{babel}
\usepackage{hyperref} 
\usepackage{enumitem}
\usepackage{caption}
\usepackage{subcaption}
\usepackage{tikz}
\usetikzlibrary{decorations.markings}

\hypersetup{
    colorlinks=false,
    pdfborder={0 0 0},
}

\newtheorem{theorem}{Theorem}[section]
\newtheorem{lemma}[theorem]{Lemma}
\newtheorem{proposition}[theorem]{Proposition}
\newtheorem{corollary}[theorem]{Corollary}
\newtheorem{conjecture}[theorem]{Conjecture}

\theoremstyle{definition}

\newtheorem{examplex}[theorem]{Example}
\newtheorem{algorithm}[theorem]{Algorithm}

\newenvironment{example}
  {\pushQED{\qed}\examplex}
  {\popQED\endexamplex}

\theoremstyle{remark}
\newtheorem{remark}[theorem]{Remark}
\newtheorem{remarks}[theorem]{Remarks}

\numberwithin{equation}{section}

 \allowdisplaybreaks

\newcommand{\ind}{{\mathbbm{1}}}
\newcommand{\1}[1]{{\ind\mkern -1.5mu}{\{#1\}}}
\newcommand{\2}[1]{\ind_{#1}}

\DeclareMathOperator{\IE}{\mathbb{E}}
\DeclareMathOperator{\IP}{\mathbb{P}}
\DeclareMathOperator{\Var}{\mathbb{V}ar}

\DeclareMathOperator{\tIE}{{\widetilde{\mathbb{E}}}}
\DeclareMathOperator{\tIP}{{\widetilde{\mathbb{P}}}}
\DeclareMathOperator{\tVar}{{\widetilde{\mathbb{V}}ar}}
\DeclareMathOperator{\tCov}{{\widetilde{\mathbb{C}}ov}}

\newcommand{\eps}{\varepsilon}
\newcommand{\re}{{\mathrm{e}}}

\newcommand{\ud}{{\mathrm d}}

\newcommand{\R}{{\mathbb R}}
\newcommand{\Z}{{\mathbb Z}}
\newcommand{\N}{{\mathbb N}}
\newcommand{\ZP}{{\mathbb Z}_+}
\newcommand{\RP}{{\mathbb R}_+}

\newcommand{\as}{\ \text{a.s.}}

\newcommand{\bbX}{{\mathbb X}}

\newcommand{\cF}{{\mathcal F}}

\newcommand{\cN}{{\mathcal N}}

\newcommand{\tW}{{\widetilde W}}
\newcommand{\tX}{{\widetilde X}}

\newcommand{\hv}{{\hat v}}
\newcommand{\hrho}{{\hat \rho}}

\newcommand{\tod}{\overset{\mathrm{d}}{\longrightarrow}}

\makeatletter
\def\namedlabel#1#2{\begingroup  
    (#2)%
    \def\@currentlabel{#2}%
    \phantomsection\label{#1}\endgroup
}
\makeatother

\newlist{myenumi}{enumerate}{10}
\setlist[myenumi]{leftmargin=0pt, labelindent=\parindent, listparindent=\parindent, labelwidth=0pt, itemindent=!, itemsep=1pt, parsep=4pt}

\newlist{thmenumi}{enumerate}{10}
\setlist[thmenumi]{leftmargin=0pt, labelindent=\parindent, listparindent=\parindent, labelwidth=0pt, itemindent=!}

\title{Dynamics of finite inhomogeneous particle systems with exclusion interaction}

\author{Vadim Malyshev\footnote{\raggedright Faculty of Mechanics and Mathematics, Lomonosov Moscow State University, Moscow, Russia.} \and Mikhail Menshikov\footnote{\raggedright Department of Mathematical Sciences, Durham University, Durham, UK. \href{mailto:mikhail.menshikov@durham.ac.uk}{\texttt{mikhail.menshikov@durham.ac.uk}}, \href{mailto:andrew.wade@durham.ac.uk}{\texttt{andrew.wade@durham.ac.uk}}.}
 \and Serguei Popov\footnote{Centro de Matem\'atica, University of Porto, Porto, Portugal. \href{mailto:serguei.popov@fc.up.pt}{\texttt{serguei.popov@fc.up.pt}}.} \and Andrew Wade\footnotemark[2]}
\date{1 October 2023}

\begin{document}
\maketitle

\begin{quote}
{{\bf Dedication.}}{~The first version of this paper was completed in June~$2022$. Vadim Alexandrovich Malyshev passed away on 30~September 2022.
The other authors, both long-term and recent collaborators, 
would like to express their deep gratitude for Vadim's influence in their mathematical lives, and 
dedicate this article to his memory. }
\end{quote}

\begin{abstract}
We study finite particle systems on the one-dimensional integer lattice, where each particle performs a continuous-time nearest-neighbour random walk, with jump rates intrinsic to each particle, subject to an exclusion interaction which suppresses jumps that would lead to more than one particle occupying any site. We show that the particle jump rates determine explicitly a unique partition of the system into maximal stable sub-systems, and that this partition can be obtained by a linear-time algorithm using only elementary arithmetic. The internal configuration of each stable sub-system possesses an explicit product-geometric limiting distribution, and the location of each stable sub-system obeys a strong law of large numbers with an explicit speed; the characteristic parameters of each stable sub-system are simple functions of the rate parameters for the corresponding particles. For the case where the entire system is stable, we provide a central limit theorem describing the fluctuations around the law of large numbers. Our approach draws on ramifications, in the exclusion context, of classical work of Goodman and Massey on partially-stable Jackson queueing networks.
\end{abstract}


\medskip


\section{Introduction}
\label{sec:intro}

The exclusion process is a prototypical model of non-equilibrium statistical mechanics, representing the dynamics of a  lattice gas
with hard-core interaction between particles, originating in the mathematical literature with~\cite{spitzer} and in the applied literature with~\cite{mgp}.
In the present paper, we consider systems of  $N+1$ particles on the one-dimensional integer lattice $\Z$, performing continuous-time, nearest-neighbour random walks with exclusion interaction,
in which each particle possesses arbitrary finite positive jumps rates.  
The configuration space of the system is $\bbX_{N+1}$, where, for $n \in \N := \{1,2,3,\ldots\}$,
\begin{equation}
\label{eq:configuration-space}
\bbX_n  := \{ (x_1, \ldots, x_{n} ) \in \Z^{n} : x_1 < x_2 < \cdots < x_{n} \} .  \end{equation}
The exclusion constraint means that there can be at most one particle at any site of~$\Z$ at any given time.
The dynamics of the interacting particle system are described by a time-homogeneous, continuous-time Markov chain on $\bbX_{N+1}$,
specified by non-negative rate parameters $a_1,b_1, \ldots, a_{N+1},b_{N+1}$. The $i$th particle (enumerated left to right)
attempts to make a nearest-neighbour jump to the left at rate~$a_i$. If, when the corresponding exponential clock rings,
the site to the left is unoccupied, the jump is executed and the particle moves, but if the destination is occupied by another particle, 
the attempted jump is suppressed and the particle does not move (this is the exclusion rule). 
Similarly, the $i$th particle attempts to jump to the right at rate $b_i$,
subject to the exclusion rule. 
The exclusion rule ensures that the left-to-right order of the particles is always preserved. See Figure~\ref{fig:particles} for a schematic.
 
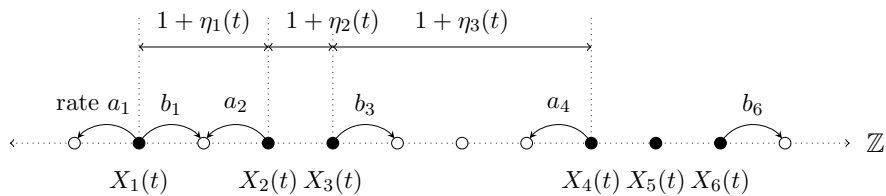
\begin{figure}[b]
\centering
\scalebox{0.85}{
 \begin{tikzpicture}[domain=0:1, scale=1.0]
\draw[dotted,<->] (0,0) -- (13,0);
\node at (13.4,0)       {$\Z$};
\draw[black,fill=white] (1,0) circle (.5ex);
\draw[black,fill=black] (2,0) circle (.5ex);
\node at (2,-0.6)       {\small $X_1 (t)$};
\draw[black,fill=white] (3,0) circle (.5ex);
\draw[black,fill=black] (4,0) circle (.5ex);
\node at (4,-0.6)       {\small $X_2 (t)$};
\draw[dotted] (2,0) -- (2,2);
\draw[dotted] (4,0) -- (4,2);
\draw[black,<->] (2,1.5) -- (4,1.5);
\node at (3,1.9)       {\small $1+\eta_1 (t)$};
\draw[black,fill=black] (5,0) circle (.5ex);
\node at (5,-0.6)       {\small $X_3 (t)$};
\draw[dotted] (5,0) -- (5,1.7);
\draw[black,<->] (4,1.5) -- (5,1.5);
\node at (5,1.9)       {\small $1+\eta_2 (t)$};
\draw[black,fill=white] (6,0) circle (.5ex);
\draw[black,fill=white] (7,0) circle (.5ex);
\draw[black,fill=white] (8,0) circle (.5ex);
\draw[black,fill=black] (9,0) circle (.5ex);
\node at (9,-0.6)       {\small $X_4 (t)$};
\draw[dotted] (9,0) -- (9,2);
\draw[black,<->] (5,1.5) -- (9,1.5);
\node at (7,1.9)       {\small $1+\eta_3 (t)$};
\draw[black,fill=black] (10,0) circle (.5ex);
\node at (10,-0.6)       {\small $X_5 (t)$};
\draw[black,fill=black] (11,0) circle (.5ex);
\node at (11,-0.6)       {\small $X_6 (t)$};
\draw[black,fill=white] (12,0) circle (.5ex);
\node at (1.3,0.6)       {\small rate $a_1$};
\node at (2.5,0.6)       {\small $b_1$};
\draw[black,->,>=stealth] (2,0) arc (30:141:0.58);
\draw[black,->,>=stealth] (4,0) arc (30:141:0.58);
\node at (3.5,0.6)       {\small $a_2$};
\node at (5.5,0.6)       {\small $b_3$};
\draw[black,->,>=stealth] (9,0) arc (30:141:0.58);
\node at (8.5,0.6)       {\small $a_4$};
\node at (11.5,0.6)       {\small $b_6$};
\draw[black,->,>=stealth] (2,0) arc (150:39:0.58);
\draw[black,->,>=stealth] (5,0) arc (150:39:0.58);
\draw[black,->,>=stealth] (11,0) arc (150:39:0.58);
\end{tikzpicture}}
\caption{Schematic of the model in the case of $N+1 = 6$ particles, illustrating some of the main notation. Filled circles represent particles, empty circles represent unoccupied lattice sites, and
directed arcs represent admissible transitions, with exponential rates indicated.}
\label{fig:particles}
\end{figure}

Let $X (t) = (X_1 (t), \ldots, X_{N+1} (t) ) \in \bbX_{N+1}$ be the configuration of the Markov chain at time $t \in \RP := [0,\infty)$,
started from a fixed initial configuration $X(0) \in \bbX_{N+1}$. Denote  the number of empty sites (\emph{holes})
between particles $i$ and $i+1$ at time~$t \in \RP$ by
\begin{equation}
\label{eq:eta-def}
\eta_i (t) := X_{i+1}(t)-X_i(t)-1, 
\text{ for } i \in [N] := \{1,2,\ldots, N\},
\end{equation}
so $\eta_i (t) \in \ZP := \{ 0,1,2,\ldots \}$. An equivalent description of the system is captured by 
the continuous-time Markov chain $\xi := ( \xi (t) )_{t \in \RP}$, where 
\begin{equation}
\label{eq:xi-def}
\xi (t) := (X_1 (t), \eta_1(t), \ldots, \eta_N (t) ) \in \Z \times \ZP^N .\end{equation}
An important  fact is that $\eta := (\eta (t))_{t \in \RP}$, where $\eta(t) := (\eta_1(t), \ldots, \eta_N (t) ) \in \ZP^N$,
is \emph{also} a continuous-time Markov chain, which can  be represented via a \emph{Jackson network} of~$N$ queues of M/M/1 type (we explain this is Section~\ref{sec:jackson}). 
This is the reason that 
we take $N+1$ particles. The process $\eta$ can also be interpreted as a generalization, with site-dependent rates and emigration and immigration at the boundaries, 
of the \emph{zero-range process} on the finite set $[N]$, with $\eta(t)$ being the vector of
particle occupancies at time $t$. 

The main contribution of the present paper is to characterize the long-term dynamics of the particle system. In particular,
we give a complete and explicit classification with regards to stability (which subsets of particles are typically relatively mutually close) and law of large numbers behaviour (particle speeds),
as well as some results on fluctuations in the case where the whole system is stable. 
Our main results are stated formally in Section~\ref{sec:results} below; their content is summarized as follows.

Theorem~\ref{thm:main}  shows that there is a unique partition of the system, determined by the
parameters $a_i, b_i$, into
maximal stable sub-systems, which we call \emph{stable clouds}. 
Theorem~\ref{thm:main}  shows that the internal configuration of each stable cloud
possesses an explicit product-geometric limiting distribution (this is a precise sense in which the cloud is `stable'),
while  distances between particles in different clouds diverge in probability.
Theorem~\ref{thm:main} also shows 
that the location of each stable cloud obeys a strong law of large numbers
with an explicit speed; speeds of clouds are non-decreasing left to right. 
The cloud partition is characterized by a finite non-linear system,
which is the translation to our setting of the classical Goodman and Massey~\cite{GM84} equations for Jackson networks.
Moreover, in Theorem~\ref{thm:algorithm} we show that the cloud partition can be obtained by a linear-time algorithm of elementary steps, streamlining the Jackson algorithm from~\cite{GM84}
by exploiting the additional structure of our model. For the case where the entire system is stable, i.e., there is a single cloud, 
Theorem~\ref{thm:clt} gives a central limit theorem describing the fluctuations around the law of large numbers;
this time, the foundational result for Jackson networks is a general central limit theorem of Anantharam and Konstantopoulos~\cite{ak}.
 Section~\ref{sec:results} presents these results, and several corollaries and examples, after introducing the necessary notation. 
First we indicate some relevant previous work.

Many aspects of interacting particle systems and their links to (finite and infinite) queueing networks are well understood,
with much attention on systems of infinitely many particles.
The exclusion and zero-range processes were introduced in a fundamental paper of Spitzer~\cite{spitzer}.
The earliest explicit mention of the link between the exclusion process and Jackson's results for queueing networks that we have seen is by Kipnis~\cite[p.~399]{kipnis}, in the case of homogeneous rates,
where the connection to the zero-range process is also given; see also~\cite{andjel,bfl,fpv}. Versions of these connections have facilitated much subsequent progress. 
For example, connections between variations on the totally asymmetric exclusion processes (TASEP), in which particles can only move to the right,
 and tandem queueing networks, in which customers are routed through a sequence of queues in series, are explored in~\cite{sepp,fm,kpsCMP}.
Other relevant work includes~\cite{arratia,khmelev,srinivasan}. 
Aspects of exclusion systems on $\Z$ with finitely many particles have been studied in the symmetric case~\cite{andjel2013}, and in 
totally asymmetric systems with different jump rates~\cite{rs}; 
exclusion interaction on finite lattice segments has also been studied~(see e.g.~\cite{cw}).

The authors are not aware of previous work on the decomposition into stable sub-systems of finite exclusion systems on $\Z$,
with general nearest-neighbour transition rates. While the connection between the exclusion process and Jackson networks is well known,
we have not seen the important work of Goodman and Massey~\cite{GM84} interpreted in the exclusion context before.
In the queueing context, the results of~\cite{GM84} characterize systems in which some but not all queues are stable;
this phenomenon has subsequently become known in the queueing literature as `partial stability'~\cite{afsw,bjl,avra}. Our main result
can thus be seen as a classification of partial stability for finite exclusion systems.

Yet another interpretation of $\eta$ is as a random walk on the orthant~$\ZP^N$ with boundary reflections.
For $N \in \{1,2\}$, there are exhaustive criteria for classifying stability of such walks (see e.g.~\cite{FMM,fim})
in terms of readily accessible quantities (first and second moments of increments).
For $N \in \{ 3, 4\}$ the generic classification is available, but requires precise knowledge of quantities 
which are hard to compute,
namely stationary distributions for lower-dimensional projections~\cite{FMM,ignatyuk}.
For $N \geq 5$, additional complexity arising from the structure of
high-dimensional dynamical systems means that the generic case
is intractable~\cite{gamarnik}.
In the present paper, we demonstrate that the dynamics of the particle system admits a complete, and explicit, stability description for any $N$,
demonstrating the special place of these models in the general theory.

In a continuum setting, there is an interesting comparison between our results and 
an extensive existing literature
on \emph{diffusions with rank-based interactions},
including the \emph{Atlas model} of financial mathematics and its relatives~\cite{pp,bfk,cdss,ik,ipbkf,kpsAIHP,tsai,sarantsevAIHP,sarantsevEJP}.
As we explain in Section~\ref{sec:discussion},
where we describe the continuum setting in more detail, 
the classification of the asymptotic behaviour of these continuum models 
has so far been limited to deciding whether the entire system is stable (i.e., a single stable cloud).
We believe that the ideas in the present paper indicate a route to obtaining more detailed
information (such as the full cloud decomposition) in these continuum models. We aim to address some of these questions in future work.
A direct comparison between our model and the continuum model is
not obvious, as there is a sense in which the collision mechanism in the continuum model
is \emph{elastic}, whereas the exclusion mechanism is \emph{inelastic}, but we describe in Section~\ref{sec:discussion}
an elastic modification of the particle system that bears a closer analogy with the continuum model. In addition,
aspects of the comparison go beyond analogy, as it is known that in certain parameter and scaling regimes, the continuum model serves as a scaling limit for certain discrete models; see~\cite[\S 3]{kpsAIHP},
and our Section~\ref{sec:discussion} below. Many of these models admit versions with infinitely-many particles (e.g.~\cite{sarantsevAIHP,sarantsevEJP}), which in certain cases reside in the famous KPZ universality class~\cite{wfs}.

\section{Main results}
\label{sec:results}

To state our main results, we define some quantities associated with our system, depending on (subsets of) the parameters $a_i, b_i$. For our main results, we will assume that at least all the~$b_i$ are positive.

\medskip
\begin{description}
\item\namedlabel{ass:positive-rates}{$\text{A}$} 
Suppose that $0 \leq a_i < \infty$ and $0 < b_i < \infty$ for all $i \in [N+1]$.
\end{description}
\medskip

Note that, by reversing the ordering of the particles (i.e., mapping~$\Z$ to $-\Z$),
we can swap the roles of the $a_i$ and $b_i$ in~\eqref{ass:positive-rates} and in the formulas that follow.
For $\ell \in \Z$, $m \in \N$,  call the set
\[ [ \ell ; m ] := \{ \ell, \ell+1, \ldots, \ell+m-1 \} \]
of $m$ consecutive integers a \emph{discrete interval};
implicit is that a discrete interval is non-empty. 
In the case $\ell = 1$, we set $[n]:= [ 1; n] := \{1,2,\ldots, n \}$ 
for $n \in \N$.
Given a discrete interval $I = [ \ell; m] \subseteq [N+1]$, which represents the 
	particles whose labels are in~$I$,
define
\begin{equation}
\label{eq:alpha-beta-def}
\begin{split}
\alpha (I) & := \alpha ( \ell ; m ) := \prod_{u=\ell}^{\ell+m-1} \frac{a_{u}}{b_{u}}, \text{ and } \\
\beta (I) & := \beta (\ell; m) := \frac{1}{b_{\ell+m-1}} \sum_{v=0}^{m-1} \prod_{u=1}^{v} \frac{a_{\ell+m-u}}{b_{\ell+m-u-1}};
\end{split}
\end{equation}
here and throughout the paper, the convention is that an empty sum is~$0$ and an empty product is~$1$. 
Then define
\begin{equation}
\label{eq:hv-def}
\hv ( I ) := \hv ( \ell ; m ) := \frac{1 - \alpha (I)}{\beta(I)} .
\end{equation}

The quantities $\alpha$, $\beta$ defined in~\eqref{eq:alpha-beta-def}
arise as solutions of certain balance
equations associated with a dual random walk derived from the dynamics of a tagged empty site of $\Z$ (a hole).
To describe this walk it is more convenient to work in the queueing setting, where the hole corresponds to a customer in the queueing network,
and we can impose a priority service policy to make sure the tagged customer is always served;
the resulting routing matrix~$P$ given by~\eqref{eq:P-def} below is then the source of the formulas for $\alpha$, $\beta$. The interpretation of $\hv (I)$ will be as a \emph{speed} for a putative stable sub-collection of particles.
For example, $\hv (\ell; 1) = b_\ell - a_\ell$ is the intrinsic speed of a singleton particle in free space;
whether its long-term dynamics matches this speed depends on its interaction with the rest of the system.
For general $m \in \N$, definition~\eqref{eq:hv-def} gives the formula
\begin{align}
\label{eq:hv-l-m}
 \hv (\ell; m) = \frac{\displaystyle 1 - \frac{a_\ell \cdots a_{\ell+m-1}}{b_\ell \cdots b_{\ell+m-1}}}{\displaystyle \frac{1}{b_{\ell+m-1}} + \frac{a_{\ell+m-1}}{b_{\ell+m-1} b_{\ell+m-2}} + \cdots + \frac{ a_{\ell+m-1} \cdots a_{\ell+1}}{b_{\ell+m-1} \cdots b_\ell}} .
\end{align}
		For a discrete interval  $I = [ \ell; m] \subseteq [N+1]$ with $m \geq 2$,
define for $j \in [\ell; m-1]$,
		\begin{equation}
		\label{eq:rho-I-def}
		\hrho_I (j) := \alpha ( \ell ; j+1 - \ell ) + \beta ( \ell ; j+1 - \ell ) \hv ( \ell ; m),
\end{equation} 
where $\alpha, \beta$ are given by~\eqref{eq:alpha-beta-def}, and $\hv$ is given by~\eqref{eq:hv-def}.
The interpretation of $\hrho_I (j)$ is as a stability parameter; it corresponds in the language of
queueing theory to a \emph{workload} for the queue associated with gap~$j$ in the system~$I$. See Section~\ref{sec:jackson} for elaboration of the queueing interpretation.
		
An ordered $n$-tuple $\Theta = (\theta_1, \ldots, \theta_n)$ of discrete intervals $\theta_i \subseteq [N+1]$
is called an \emph{ordered partition} of $[N+1]$ if (i) $\theta_i \cap \theta_j = \emptyset$ for all $i \neq j$; 
(ii) $\cup_{i=1}^n \theta_i = [N+1]$; and (iii) for every $\ell, r \in [n]$
with $\ell < r$, every $i \in \theta_{\ell}$ and $j \in \theta_{r}$ satisfy $i < j$. Here $n =: \vert \Theta \vert$ is the number of \emph{parts} in the partition.
Write $\theta \in \Theta$ if $\theta$ is one of the parts of $\Theta$. 
Given 
an ordered partition $\Theta = (\theta_1, \ldots, \theta_n)$,
we write $\Theta^\star := ( \theta_j : \vert \theta_j \vert \geq 2 )$
for the ordered non-singletons; if $\vert \theta_j \vert = 1$ for all $j$, we set $\Theta^\star := \emptyset$.
	 
For $I \subseteq [N+1]$ a  discrete interval, define
	\begin{equation}
	\label{eq:R-def}
	R_I (t) := \max_{i \in I} X_i (t) - \min_{i \in I} X_i (t) , \text{ for } t \in \RP, \end{equation}
	the total distance between the extreme particles indexed by~$I$. 
	For $I = [\ell;m]$ a  discrete interval with $m \geq 2$ elements, write $I^\circ := [\ell;m-1]$
for the discrete subinterval that omits the maximal element.
	If $\vert I\vert \geq 2$, define
	\begin{equation}
	\label{eq:partial-eta-def}
	\eta_I (t) := ( \eta_j (t) : j \in I^\circ ) ,
	\end{equation}
	the vector of particle separations restricted to particles in $I$.
	 
	Part of our result will be to identify an ordered partition $\Theta = (\theta_1, \ldots, \theta_n)$ of $[N+1]$
	in which each $\theta \in \Theta$ is  a \emph{stable cloud} in the long-term dynamics of the particle system. 
	Stability means that the relative displacements of the particles indexed by $\theta$ are ergodic, in a sense made precise in the statement of Theorem~\ref{thm:main} below, which includes an explicit limit distribution
	for $\eta_{\theta} (t)$; moreover, each cloud travels at an asymptotic speed. To describe the limit distribution for the displacements within each cloud, 
	we define, for  $I \subseteq [N+1]$  a   discrete interval with $\vert I\vert=1+k \geq 2$ elements,
	\begin{equation}
	\label{eq:limit-distribution-component}
	\varpi_I ( z_1, \ldots, z_{k} ) := \prod_{j \in I^\circ} \hrho_I (j)^{z_j} (1-\hrho_I(j)) , \text{ for } (z_1,\ldots, z_{k}) \in \ZP^{k} , \end{equation}
	where $\hrho_I$ is defined at~\eqref{eq:rho-I-def}. 
For $A \subseteq \ZP^k$, set $\varpi_I ( A) := \sum_{z \in A} \varpi_I (z)$. Then, if $0 < \hrho_I(j) < 1$ for every $j$,
$\varpi_I$ is a non-vanishing probability distribution on $\ZP^k$.
	
	The stable cloud decomposition is encoded in a \emph{general traffic equation},
arising from the connection between our model and a Jackson network
of queues in series with nearest-neighbour routing. We describe this connection in detail in Section~\ref{sec:jackson};
for now we introduce the notation needed to formulate the general traffic equation~\eqref{eq:general-balance}.
Define
\begin{equation}
\label{eq:mu-def}
\mu_i := b_i + a_{i+1}, \text{ for }
i \in [N]. \end{equation}
If $N=1$, define $\lambda_1 := a_1 + b_2$; else, for $N \geq 2$, define
\begin{equation}
\label{eq:lambda-def}
\lambda_1 := a_1, ~ \lambda_{N} := b_{N+1}, \text{ and }
\lambda_j := 0 \text{ for } 2 \leq j \leq N-1.
\end{equation} 
Also define the  matrix $P:= (P_{i,j})_{i,j \in [N]}$ by
\begin{equation}
\label{eq:P-def}
\begin{split}
P_{i,i-1} & := \frac{b_i}{\mu_i} = \frac{b_i}{b_i+a_{i+1}}, \text{ for } 2 \leq i \leq N; \\
P_{i,i+1} & := \frac{a_{i+1}}{\mu_i} = \frac{a_{i+1}}{b_i+a_{i+1}}, \text{ for } 1 \leq i \leq N-1;
\end{split}
\end{equation}
and $P_{i,j}:=0$ for all $i,j$ with $\vert i-j\vert \neq 1$. For vectors $x= (x_i)_{i \in [N]}$ and $y = (y_i)_{i \in [N]}$ in $\R^N$, write $x \wedge y$ for the
vector $(x_i \wedge y_i)_{i \in [N]}$. The general traffic equation is the matrix-vector equation for a vector $\nu \in \R^N$ given by 
\begin{equation}
\label{eq:general-balance}
\nu = ( \nu \wedge \mu ) P + \lambda,
\end{equation}
where $\mu = (\mu_i)_{i \in [N]}$ and $\lambda = (\lambda_i)_{i \in [N]}$. In~\eqref{eq:general-balance}
 and elsewhere we view vectors as column vectors when necessary. In the Jackson network context, the quantities $\lambda,\mu$, and $P$ represent
arrival rates of customers, service rates, and routing probabilities, respectively; see Section~\ref{sec:jackson} for a precise description.
Now we can state our first main result.
	
		\begin{theorem}
	\label{thm:main}
	Let $N \in \N$ and suppose that~\eqref{ass:positive-rates} holds.
	There exists a unique solution $\nu = (\nu_i)_{i \in [N]}$
	to the general traffic equation~\eqref{eq:general-balance}. Define $\rho_i := \nu_i / \mu_i$
	for every $i \in [N]$. 
	Then there is a unique 
	ordered partition $\Theta = (\theta_1, \ldots, \theta_n)$ of $[N+1]$,
	which we call the \emph{cloud partition}, 
	such that, with $\hrho_\theta (j)$ as defined at~\eqref{eq:rho-I-def},
	\begin{equation}
	\label{eq:rho-rho-hat}
	\text{for every } \theta \in \Theta^\star \text{ and all } j \in \theta^\circ, ~ \rho_j = \hrho_\theta (j) \in (0,1),
\end{equation}
	and $\rho_{\max\theta} \geq 1$ for every $\theta \in \Theta$ with $\max \theta \leq N$.
The following stability statements hold.	
\begin{thmenumi}[label=(\roman*)]
\item 
\label{thm:main-i}
 For every $\theta \in \Theta^\star$, take $A_\theta \subseteq \ZP^{k(\theta)}$, where $k(\theta) := \vert \theta \vert -1 \geq 1$. Then
\[ \lim_{t \to \infty} \frac{1}{t} \int_0^t \ind \left( \bigcap_{\theta \in \Theta^\star} \left\{ \eta_{\theta} (s) \in A_\theta \right\} \right) \ud s = 
\lim_{t \to \infty} \IP \left( \bigcap_{\theta \in \Theta^\star} \left\{ \eta_{\theta} (t) \in A_\theta \right\} \right)
 = \prod_{\theta \in \Theta^\star} \varpi_{\theta} (A_\theta) , \]
the first limit holding a.s.~and in $L^1$, where~$\varpi_\theta$ is defined by~\eqref{eq:limit-distribution-component}. 
Moreover, 
\begin{equation}
\label{eq:cloud-size}
 \lim_{t \to \infty} \IE R_\theta (t) = \sum_{j \in \theta^\circ} \frac{1}{1-\hrho_\theta (j) } , \text{ for every } \theta \in \Theta^\star. \end{equation}
\item 
\label{thm:main-ii} On the other hand, for $1 \leq \ell < r \leq n$, we have for 
every $i \in \theta_\ell$ and $j \in \theta_r$,
\begin{equation}
\label{eq:cloud-separation}
 \lim_{t \to \infty} \IP \left[ \vert X_j (t) - X_i (t) \vert  \leq B \right] = 0 , \text{ for every }
B \in \RP.
\end{equation}
\item 
\label{thm:main-iii}
For every $i \in [N+1]$, define $v_i$ by $v_i = \hv (\theta)$ where $i \in \theta \in \Theta$, and $\hv(\theta)$ is as defined at~\eqref{eq:hv-def}.
Then $-\infty <  v_1 \leq \cdots \leq v_n < \infty$, and
	\[ \lim_{t \to \infty} \frac{X_i (t)}{t} = v_i, \as, \text{ for every } i \in [N+1] .\]
	Moreover, $v_{i+1} - v_i  =  (\rho_i -1)^+ ( b_i + a_{i+1} )$ for all $i \in [N]$, where $x^+ := x \1 { x \geq 0}$, $x \in \R$.
	\end{thmenumi}
	\end{theorem}
	
	\begin{remarks}
	\phantomsection
\label{rems:main-theorem}
\begin{myenumi}[label=(\alph*)]
\setlength{\itemsep}{0pt plus 1pt}
\item
	Statement~\ref{thm:main-i} is our ergodicity property for the stable clouds; 
	note it has nothing to say about singletons in the cloud partition~$\Theta$. 
Statement~\ref{thm:main-ii} says that distances between particles in different clouds diverge in probability, or, in other words,
	the stable clouds given by $\Theta$ are maximal.  Statement~\ref{thm:main-iii} says that each stable cloud in $\Theta$ possesses an asymptotic \emph{speed}
	at which all particles in that cloud travel.
	\item
	If $\theta_\ell, \theta_r$ with $1 \leq \ell < r \leq \vert \Theta \vert$ are two distinct stable clouds, and if the strict inequality $\hv(\theta_\ell) < \hv (\theta_r)$
	holds for the corresponding speeds in part~\ref{thm:main-iii}, then $\Delta_{\ell,r} (t) := \min_{i \in \theta_r} X_i (t) -  \max_{i \in \theta_\ell} X_i (t)$ 
	satisfies $\lim_{t \to \infty} t^{-1} \Delta_{\ell,r} (t) = \hv(\theta_r) - \hv(\theta_\ell) >0$, a.s.,
	which is much stronger than the statement in~\ref{thm:main-ii}. However, if $\hv(\theta_\ell) = \hv(\theta_r)$, it may be the case that $\liminf_{t \to \infty} \Delta_{\ell,r} (t) < \infty$, a.s.:
see Example~\ref{ex:all-equal} below.
\end{myenumi}
\end{remarks}

Theorem~\ref{thm:main} describes the asymptotic behaviour of the particle system through
the cloud partition $\Theta$ and the formulas~\eqref{eq:hv-l-m} and~\eqref{eq:rho-I-def} for the $\hv$ and $\hrho$.
Partition $\Theta$ is characterized via the solution $\nu$ of the non-linear
system~\eqref{eq:general-balance}; those $i$ for which $\rho_i \geq 1$ mark boundaries between successive parts of $\Theta$.
While it is not, in principle, a difficult task for a computer to solve the system~\eqref{eq:general-balance},
we present below an algorithm for obtaining $\Theta$ via 
a sequence of comparisons involving only applications of the formula~\eqref{eq:hv-l-m},
without directly computing the solution to~\eqref{eq:general-balance}.

There is a closely related algorithm due to Goodman and Massey~\cite{GM84}, in the general Jackson network context.
In our setting we can exploit the linear structure 
to produce an algorithm that needs only the formula~\eqref{eq:hv-l-m}, involving elementary arithmetic (a similar simplification takes place for the algorithm of~\cite{GM84}).
 The structure of our algorithm (and Theorem~\ref{thm:block-merge} below from which it is derived) is of
additional interest, as it provides some intuition as to
how stable clouds are formed.
Roughly, the algorithm goes as follows: start from the ordered partition of all singletons
as candidate stable clouds, and successively merge any candidate stable clouds in which the intrinsic speeds, as computed 
by~\eqref{eq:hv-l-m}, are such that the speed of the candidate cloud to the left exceeds the speed of the candidate cloud to the right.
In the Jackson network context, the intuition-bearing
 quantities in the algorithm of~\cite{GM84} are
the net growth rates of the queues, which are \emph{differences} of our speeds; specifically, $v_{i+1} - v_i= (\nu_i - \mu_i)^+$ is the net growth rate of queue $i$ (see Lemma~\ref{lem:speeds-algebra} below).

\begin{theorem}
\label{thm:algorithm}
	Let $N \in \N$ and suppose that~\eqref{ass:positive-rates} holds.
Then Algorithm~\ref{alg:partition} below produces the
unique partition $\Theta$ featuring in Theorem~\ref{thm:main}.
\end{theorem}

\begin{algorithm}
\label{alg:partition}
The algorithm takes as input data $N \in \N$ and the parameters $a_i, b_i \in (0,\infty)$, $i \in [N+1]$,
and   outputs an ordered partition $\Theta$ of $[N+1]$.
\begin{enumerate}
	\item Initialize with~$\kappa =0$ and $\Theta^0 := ( \{1\}, \{2\}, \ldots, \{N+1\} )$,
the partition of singletons.
\item Given the ordered partition
  $\Theta^\kappa = (\theta_1^\kappa, \ldots, \theta_n^\kappa)$
with $n = \vert \Theta^\kappa \vert$, compute $\hv (\theta^\kappa_i)$ for every $1 \leq i \leq n$, using~\eqref{eq:hv-l-m}.
	\item If $\hv ( \theta^\kappa_1 ) \leq \hv (\theta^\kappa_2 ) \leq \cdots \leq \hv (\theta^\kappa_n )$, where $\vert \Theta^\kappa \vert = n$, then set $\Theta = \Theta^\kappa$ and {\sc stop}.
	\item Otherwise, perform the following update procedure to get $\Theta^{\kappa+1}$.
 Let $J := \{ j \in [n-1] : \hv ( \theta^\kappa_j ) > \hv (\theta^\kappa_{j+1} ) \}$, where $n = \vert \Theta^\kappa \vert \geq 2$. 
Choose some $j \in J$. Then define
\[ \theta^{\kappa+1}_i := \begin{cases} 
\theta^{\kappa}_i &\text{if } 1 \leq i \leq j-1, \\
\theta^{\kappa}_j \cup \theta^{\kappa}_{j+1} &\text{if } i = j, \\
\theta^{\kappa}_{i+1} &\text{if } j+1 \leq i \leq n-1, \end{cases} \]
and take $\Theta^{\kappa+1} = (\theta^{\kappa+1}_1, \ldots, \theta^{\kappa+1}_{n-1} )$.
Iterate $\kappa \mapsto \kappa+1$ and return to {\sc Step 2}.
\end{enumerate}
	\end{algorithm}

\begin{remarks}
	\phantomsection
\label{rems:algorithm}
\begin{myenumi}[label=(\alph*)]
\setlength{\itemsep}{0pt plus 1pt}
\item
If $\vert\Theta^\kappa\vert =1$ at any point, then {\sc Step 3} will terminate the algorithm; thus whenever {\sc Step 2} is executed, one has $n = \vert \Theta^\kappa \vert \geq 2$.
Similarly, the set $J$ in {\sc Step 4} will always be non-empty.
\item
The merger executed in {\sc Step 4} reduces by one the number of parts in the partition, so Algorithm~\ref{alg:partition}
will terminate in at most $N$ iterations.
\item
One can in fact perform the merger in {\sc Step 4} at \emph{every} $j \in J$, rather than picking just one; this follows from Theorem~\ref{thm:block-merge} below. However, for simplicity of presentation
we only perform a single pairwise merger per iteration in the description above.
\item The algorithm of Goodman and Massey~\cite[p.~863]{GM84} is similar, in that it also requires at most $N$ steps, while each step requires inverting a matrix, which, in our case, is tridiagonal
and so provides formulas comparable to~\eqref{eq:hv-l-m}; the sequence of cloud mergers in Algorithm~\ref{alg:partition} is a little more adapted to the linear structure of our setting.
\end{myenumi}
\end{remarks}

The proofs of Theorem~\ref{thm:main} and~\ref{thm:algorithm} are given in Section~\ref{sec:proofs} below.
Next we state several corollaries to these two results: the proofs of these also appear in Section~\ref{sec:proofs}.
The first two corollaries pertain to the extreme cases of a single stable cloud, and a system
	in which each particle constitutes its own stable cloud.
	
	\begin{corollary}
	\label{cor:singletons}
	Let $N \in \N$ and suppose that~\eqref{ass:positive-rates} holds.
	The partition $\Theta$ consists only of singletons if and only if 
	\begin{equation}
	\label{eq:singleton-condition}
	b_1 - a_1 \leq b_2 - a_2 \leq \cdots \leq b_{N+1} - a_{N+1} .\end{equation}
	In the case where~\eqref{eq:singleton-condition} holds, we have that
	$\lim_{t \to \infty} t^{-1} X_i (t) = b_i - a_i$, a.s., while, for every $i \neq j$,
	$\vert X_i (t) - X_j (t) \vert \to \infty$ in probability as $t \to \infty$.
	\end{corollary}
				
		Recall that $\hrho$ is defined by~\eqref{eq:rho-I-def}. From Theorem~\ref{thm:main},
		we have that $\Theta = ([N+1])$ consists of a single stable cloud if and only if $\rho_i = \hrho_{[N+1]} (i) < 1$ for all $i \in [N]$,
		where $\rho_i = \nu_i /\mu_i$ is defined in terms of the solution $\nu$ to the general traffic equation~\eqref{eq:general-balance}; see~\eqref{eq:rho-rho-hat}.
		Furthermore, if $\rho_i < 1$ for all $i$, then~\eqref{eq:general-balance} reduces to the linear equation $\nu = \nu P + \lambda$, which can be solved explicitly (see Section~\ref{sec:jackson} below),
to give the formula~\eqref{eq:rho-stable} below. Thus we will sometimes also, with a small abuse of notation, refer to~\eqref{eq:rho-stable} as a definition of $\rho_i$ in the stable case.
				
	\begin{corollary}
	\label{cor:stable}
	Let $N \in \N$ and suppose that~\eqref{ass:positive-rates} holds.
	The partition $\Theta$ has a single part if and only if $\rho_i < 1$ for all $i \in [N]$, where,  for $i \in [N]$, 
	\begin{align}
	\label{eq:rho-stable}
	\rho_i = \hrho_{[N+1]} (i) = 
 \frac{a_1\cdots a_i}{b_1\cdots b_i}
  +\left( \frac{1}{b_i}+\frac{a_i}{b_ib_{i-1}}+
  \cdots + \frac{a_i\cdots a_2}{b_i\cdots b_1}\right) \hv_{N+1}, 
\end{align}
and where $\hv_{m} := \hv (1; m)$ is given by the $\ell=1$ version of~\eqref{eq:hv-l-m}, i.e.,
\begin{equation}
\label{eq:v-stable} 
 \hv_m := \frac{1-
\displaystyle\frac{a_1\cdots a_{m}}{b_1\cdots b_{m}}}
{\displaystyle\frac{1}{b_{m}}
 +\frac{a_{m}}{b_{m}b_{m-1}}+
  \cdots + \frac{a_{m}\cdots a_2}{b_{m}\cdots b_1}}, \text{ for } 1 \leq m \leq N+1.
\end{equation}
Moreover, if $\rho_i < 1$ for all $i \in [N]$, then for all $z_1, \ldots, z_N \in \ZP$,
\begin{equation}
\label{eq:stable-limit}
\lim_{t \to \infty}  \IP \left( \bigcap_{i=1}^N \left\{ \eta_{i} (t) = z_i \right\} \right)
 = \prod_{i=1}^N \rho_i^{z_i} (1-\rho_i) , \end{equation}
and
 $\lim_{t \to \infty} t^{-1} X_i (t) = \hv_{N+1}$, a.s., for every $i \in [N+1]$, where $\hv_{N+1}$ is given by~\eqref{eq:v-stable}.
\end{corollary}
\begin{remarks}
	\phantomsection
\label{rems:stable}
\begin{myenumi}[label=(\alph*)]
\setlength{\itemsep}{0pt plus 1pt}
\item 
A compact, but less explicit, expression of the stability condition in Corollary~\ref{cor:stable} is as $\lambda ( I - P)^{-1} < \mu$, understood componentwise, where~$\lambda, \mu, P$ are given by~\eqref{eq:mu-def}--\eqref{eq:P-def}, and $I-P$ is an invertible tridiagonal matrix;
this is the classical Jackson network stability condition translated to our model: see Proposition~\ref{prop:stable-Jackson} below.
\item
Here is a more intuitive expression of the stability condition in Corollary~\ref{cor:stable}, in the spirit of Algorithm~\ref{alg:partition}. 
From~\eqref{eq:rho-stable}, \eqref{eq:rho-I-def}, and~\eqref{eq:v-stable}, we have that $\hv_i = \hv (1;i) =  (1 - \alpha (1;i))/\beta(1;i)$ and 
$\rho_i = \rho_{[N+1]} (i) = \alpha (1 ; i) + \beta (1; i) \hv_{N+1}$ with the notation at~\eqref{eq:alpha-beta-def}. Hence
\begin{equation}
\label{eq:speeds-condition}
\rho_i < 1 \text{ if and only if } \hv_i > \hv_{N+1} ,
\end{equation}
which expresses the stability condition in terms of the 
putative intrinsic speed associated with the sub-system of particles $\{1,2,\ldots,i\}$
compared to the putative intrinsic speed of the whole system. See Section~\ref{sec:discussion} for related results in a diffusion context.
\end{myenumi}
\end{remarks}

In the case of a single stable cloud, Theorem~\ref{thm:clt} below gives a central limit theorem
to accompany Corollary~\ref{cor:stable}. First we present some illustrative examples.
		
	\begin{example}[Constant drifts]
	\label{ex:all-equal}
	If $b_i - a_i \equiv u \in \R$ for all $i \in [N+1]$,
	then~\eqref{eq:singleton-condition} clearly holds,
	and so Corollary~\ref{cor:singletons} applies, and each particle is a singleton cloud
	with $\lim_{t \to \infty} t^{-1} X_i(t) = u$, a.s.
	Moreover, for every $i \in [N]$, the gap $X_{i+1} (t) - X_i (t) = 1 + \eta_i (t)$
	satisfies $\lim_{h \to 0} h^{-1} \IE [ \eta_i (t +  h) - \eta_i (t) \mid \cF_t ] \leq 0$ on $\{ \eta_i (t) \geq 1 \}$,
	where $\cF_t = \sigma ( X(s), s \leq t )$.
	It follows from this, or the corresponding supermartingale condition for the discrete-time embedded Markov chain,
	that  $\liminf_{t \to \infty} ( X_{i+1} (t) - X_i (t) ) = 1$, a.s.~(combine, for example, Thm.~3.5.8 and Lem.~3.6.5 from~\cite{mpw}). Thus although the
	particles form singleton clouds, nevertheless every pair of consecutive particles meets infinitely often (particles meeting means that they occupy adjacent sites).
	One can ask whether, say, three or four consecutive particles meet infinitely often or not; we do not address such `recurrence or transience' questions systematically here, but cover some cases in Example~\ref{ex:all-equal2} below.
	\end{example}
	
	\begin{example}[Totally asymmetric case]
\label{ex:a-zero}
The assumption~\eqref{ass:positive-rates} permits $a_i= 0$; here is one such example.
Suppose that $a_i =0$ for all $i \in [N+1]$, but $b_i >0$ for all $i \in [N+1]$.
This is the \emph{totally asymmetric} case in which particles can jump only to the right,
but may do so at different rates: see e.g.~\cite{rs} and references therein.
Corollary~\ref{cor:stable} shows that 
the system is stable if and only if $b_i > b_{N+1}$ for all $i \in [N]$ and, moreover, if this latter condition is satisfied,
$\lim_{t \to \infty} t^{-1} X_i (t) = \hv_{N+1} = b_{N+1}$ for all $i \in [N+1]$.
\end{example}
	
	\begin{example}[Small systems]
	\label{ex:small}
	Suppose that $N=1$ (two particles). 
	Corollary~\ref{cor:stable} implies that the system is stable if and only if $b_1 - a_1 > b_2 - a_2$;
	if stable, then $\lim_{t \to \infty} t^{-1} X_1 (t) = \lim_{t \to \infty} t^{-1} X_2 (t) = \hv_2$, a.s.,
	where~\eqref{eq:v-stable}  gives $\hv_2 = (b_1 b_2 - a_1 a_2 )/ (a_2 + b_1)$.
	For a quantitative central limit theorem describing the fluctuations around this strong law, see Theorem~\ref{thm:two-particle-clt} below.
	On the other hand, if $b_1 - a_1 \leq b_2 - a_2$ then
	Corollary~\ref{cor:singletons} applies, and $\{1\}$ and $\{2\}$ are separate stable clouds.
	
	Suppose that $N=2$ (three particles). Consider the four inequalities
	\begin{align}
	\label{eq:3-particles-1} b_1 - a_1 & > b_2 - a_2 ; \\
	\label{eq:3-particles-2} b_2 - a_2 & > b_3 - a_3 ; \\
		\label{eq:3-particles-rho1}  b_1 b_2 + b_1 a_3 + a_2 a_3  & > a_1 b_2 + a_1 a_3 + b_2 b_3 ; \\
			\label{eq:3-particles-rho2} b_1 b_2 + b_1 a_3 + a_2 a_3 & > a_1 a_2 + b_1 b_3 + a_2 b_3 .
\end{align}
		In this case, the complete classification of the system is as follows.
		\begin{itemize}
		\item[(i)] If~\eqref{eq:3-particles-rho1} and~\eqref{eq:3-particles-rho2} both hold, then $\{1,2,3\}$ is a stable cloud.
		\item[(ii)] If~\eqref{eq:3-particles-1} holds but~\eqref{eq:3-particles-rho2} fails, then the stable clouds are $\{1,2\}$ and $\{3\}$.
		\item[(iii)] If~\eqref{eq:3-particles-2} holds but~\eqref{eq:3-particles-rho1} fails, then the stable clouds are $\{1\}$ and $\{2,3\}$.
		\item[(iv)] If~\eqref{eq:3-particles-1} and~\eqref{eq:3-particles-2} both fail, then the stable clouds are $\{ 1\}$, $\{2\}$, $\{3\}$.
		\end{itemize}
		This classification is exhaustive, as can be seen from the following implications:
		\begin{itemize}
		\item[(a)] If~\eqref{eq:3-particles-1}
		and~\eqref{eq:3-particles-2} both hold, then so do~\eqref{eq:3-particles-rho1}
		and~\eqref{eq:3-particles-rho2}.
		\item[(b)] If~\eqref{eq:3-particles-1} and~\eqref{eq:3-particles-rho2} hold but~\eqref{eq:3-particles-2} does not, then~\eqref{eq:3-particles-rho1} holds.
				\item[(c)] If~\eqref{eq:3-particles-2} and~\eqref{eq:3-particles-rho1} hold but~\eqref{eq:3-particles-1} does not, then~\eqref{eq:3-particles-rho2} holds.
				\end{itemize}
				To verify~(a), note for example that~\eqref{eq:3-particles-1} and~\eqref{eq:3-particles-2} imply $(b_1-a_1) (b_2+a_3) > b_2 (b_3-a_3) + a_3 (b_2-a_2) = b_2 b_3 - a_2 a_3$, which is~\eqref{eq:3-particles-rho1}.
				For~(b), note that~\eqref{eq:3-particles-1} implies $a_1 a_2 + b_1 b_3 + a_2 b_3 > a_1 ( a_2 + b_3 ) + b_2 b_3 \geq a_1 (a_3 + b_2) + b_2 b_3$ if~\eqref{eq:3-particles-2} fails,
and				then~\eqref{eq:3-particles-rho2} implies~\eqref{eq:3-particles-rho1}; (c) is similar.

To see how the classification laid out here follows from the stated results, note that, with $\rho_i$ as defined at~\eqref{eq:rho-stable}, inequalities~\eqref{eq:3-particles-rho1} and~\eqref{eq:3-particles-rho2}
		are equivalent to $\rho_1 < 1$ and $\rho_2 < 1$, respectively. Thus Corollary~\ref{cor:stable} yields~(i). On the other hand, Corollary~\ref{cor:singletons} gives~(iv).
		For (ii), notice that~\eqref{eq:3-particles-1} implies that a valid first step for Algorithm~\ref{alg:partition} is to merge $\{1\}$ and $\{2\}$, but then failure of~\eqref{eq:3-particles-rho2} 
		is equivalent to $\hv ( \{1,2\} ) \leq \hv (\{3\})$, where, by~\eqref{eq:hv-l-m},
		$\hv ( \{1,2\} ) = (b_1 b_2 - a_1 a_2)/(a_2 + b_1)$ and $\hv ( \{3\} ) = b_3 - a_3$. Similarly for~(iii). Finally, note that in the stable case~(i), the limiting speed is
	\begin{equation}
	\label{eq:3-particles-speed}
	\hv_3 = \frac{b_1 b_2 b_3 - a_1 a_2 a_3}{b_1 b_2 + b_1 a_3 + a_2 a_3 } ,\end{equation}
			by~\eqref{eq:v-stable}.
			\end{example}
		
	Here is one more corollary,
	which gives a necessary and sufficient condition
	for all the speeds $v_i$ in Theorem~\ref{thm:main}\ref{thm:main-iii} to be positive.
	
	\begin{corollary}
	\label{cor:one-direction}
	Let $N \in \N$ and suppose that~\eqref{ass:positive-rates} holds.
	The following are equivalent.
	\begin{itemize}
	\item[(i)] $\min_{i \in [N+1]} v_i > 0$.
	\item[(ii)] For every $k \in [N+1]$, $\prod_{i=1}^k a_i < \prod_{i=1}^k b_i$.
	\end{itemize}
	\end{corollary}
	
We next turn to fluctuations. For simplicity, we assume that the system is stable, i.e., the $\rho_i$
given by~\eqref{eq:rho-stable} satisfy $\rho_i <1$ for all $i \in[N]$. 
The following central limit theorem  is a consequence of  
results of Anantharam and Konstantopoulos~\cite{ak}
for counting processes of transitions in continuous-time Markov chains.
Let $\cN(\mu,\sigma^2)$ denote the normal distribution on $\R$ with mean $\mu \in \R$ and variance $\sigma^2 \in (0,\infty)$.
	
		\begin{theorem}
	\label{thm:clt}
	Let $N \in \N$ and suppose that~\eqref{ass:positive-rates} holds.
	Suppose that the $\rho_i$
given by~\eqref{eq:rho-stable} satisfy $\rho_i <1$ for all $i \in[N]$.
Let $\hv_{N+1}$ be given by~\eqref{eq:v-stable}.
Then there exists $\sigma^2 \in (0,\infty)$ such that, for all $w \in \R$ and $z_1, \ldots, z_N \in \ZP$,
\[
\lim_{t \to \infty} \IP \left[ \frac{X_1(t) - \hv_{N+1} t}{\sqrt{ \sigma^2 t}} \leq  w ,\, \eta_1 (t) = z_1, \, \ldots , \, \eta_N (t) = z_N \right] = \Phi (w) \prod_{i =1}^N \rho_i^{z_i} (1-\rho_i ) ,
\]
where $\Phi$ is the cumulative distribution function of $\cN(0,1)$.
In particular, if $Z \sim \cN (0,  \sigma^2)$, then, as $t \to \infty$,
\[ \left( \frac{X_i(t) - \hv_{N+1} t}{\sqrt{t}} \right)_{i \in [N+1]} \tod (Z , \ldots ,Z). \]
	\end{theorem}
	
	\begin{remarks}
		\phantomsection
	\label{rems:clt}
\begin{myenumi}[label=(\alph*)]
\setlength{\itemsep}{0pt plus 1pt}
\item	\label{rems:clt-a}  Theorem~\ref{thm:clt} covers the case of a single stable cloud. In the general (partially stable) case,
we anticipate that there is a joint central limit theorem in which the components associated with clouds with different
speeds are independent, while those associated with clouds of the same speed are correlated;
  in the case in which all clouds are singletons, one case of such a result is known in the context of the scaling limit
  results of~\cite{kpsAIHP}: see Section~\ref{sec:scaling} below. 
\item 	\label{rems:clt-b}
In the case $N \geq 2$, Theorem~\ref{thm:clt}
is deduced from Thm.~10 of~\cite{ak},
which gives a central limit theorem for counting processes in stable Jackson networks;
when $N \geq 2$, Thm.~11 of~\cite{ak} also gives an expression for the limiting variance~$\sigma^2$, which is in principle
computable, but somewhat involved. For $N=1$, Theorem~\ref{thm:clt}
does not fall directly under the scope of Thms.~10 and~11 of~\cite{ak},
since one must distinguish between the two different types of arrivals and departures (to the left or to the right);
nevertheless, the result can be deduced from the general approach of~\cite{ak} in the $N=1$ case too (use e.g.~their Thm.~5). 
The main additional ingredient needed to deduce Theorem~\ref{thm:clt}
from~\cite{ak} is the asymptotic independence between the fluctuations of the location of the cloud ($X_1$) and the internal displacements ($\eta$), 
which is a consequence of a separation of time scales: see Section~\ref{sec:clt} below.
\end{myenumi}
	\end{remarks}

As mentioned in Remarks~\ref{rems:clt}\ref{rems:clt-b},
the $N=1$ case of Theorem~\ref{thm:clt} is not treated explicitly in~\cite{ak}.
Thus it is of interest to give an explicit computation of $\sigma^2$ in that case.
We will use the properties of the M/M/1 queue to do so: this is the next result.
	
	\begin{theorem}
	\label{thm:two-particle-clt}
	Suppose that $N =1$,
that~\eqref{ass:positive-rates} holds, 
	and $a_1 + b_2 < a_2 + b_1$. Then
	the central limit theorem in Theorem~\ref{thm:clt} holds, with
	\begin{equation}
	\label{eq:sigma-two-particle}
	\hv_2 = \frac{b_1b_2 - a_1 a_2}{a_2 + b_1}, \text{ and } \sigma^2 = \frac{a_1 a_2 + b_1 b_2}{a_2 + b_1} .
	\end{equation}
That is, if $Z \sim \cN (0,  \sigma^2)$, with $\sigma^2$ given at~\eqref{eq:sigma-two-particle},
then, as $t \to \infty$,
\begin{equation}
\label{eq:two-particle-clt}
 \left( \frac{X_1 (t) - \hv_2 t}{\sqrt{t}} , \frac{X_2 (t) - \hv_2 t}{\sqrt{t}} \right) \tod (Z , Z). \end{equation}
	\end{theorem}
	
		We finish this section with some further examples.
		
\begin{example}[1 dog and $N$ sheep; lattice Atlas model]
\label{ex_N_sheep}
Assume that $a_1=a\in (0,1)$,
and $b_1=a_2=b_2=\cdots=b_{N+1}=1$. That is, the leftmost particle
(`the dog')
has an intrinsic drift to the right, while all other particles
(`the sheep') have no intrinsic drift.
In this case  $\rho_j$ 
 given by~\eqref{eq:rho-stable}
satisfies
\begin{equation}
\label{eq:rho_k_sheep}
 \rho_j= a + (1-a) \frac{j}{N+1} \leq \frac{N +a}{N+1} < 1, \text{ for } j \in [N] . \end{equation}
By Corollary~\ref{cor:stable} the whole system is stable,
and, by~\eqref{eq:v-stable},
 the speed of the cloud is
$\hv_{N+1} =\frac{1-a}{N+1}$. By~\eqref{eq:cloud-size} and~\eqref{eq:rho_k_sheep}, the long-run
expected size of the particle cloud satisfies
\[
\lim_{t \to \infty} \IE (X_{N+1}(t) -X_1(t) ) = \frac{N+1}{1-a}
 \left( 1+\frac{1}{2}+\cdots+\frac{1}{N} \right)
  \sim (1-a)^{-1}N\log N,
\]
as $N \to \infty$. In~\cite[pp.~191--192]{bfl} the closely related model
where $a_1 = 0$, $b_1 = 1$, and $b_i = 1 - a_i = b \neq 1/2$ for all $i \geq 2$ is treated; taking $b \to 1/2$
in~(2.30) in~\cite{bfl} recovers a version of this example.
This is a lattice relative of the continuum \emph{Atlas model}~\cite{bfk,ipbkf}, in which 
the leftmost diffusing particle (Atlas) carries the weight of the rest on its shoulders: see~Section~\ref{sec:discussion} below for
a discussion of such continuum models. 
\end{example}

\begin{example}[Sheep between two dogs, symmetric case]
\label{ex_2dogs_symm}
 We now set $a_1=b_{N+1}=a\in(0,1)$, 
and $b_1=a_2=b_2 = \cdots =a_{N+1}=1$; i.e., now the 
first and the last particles are dogs with drifts directed
inside, and the particles between those are sheep.
Then, we readily find $\rho_k=a$ for $k=1,\ldots,N$
and $\hv_{N+1} =0$ (the last fact is clearly a consequence of symmetry).
In particular, the expected size 
of the cloud is linear in~$N$, unlike in Example~\ref{ex_N_sheep}.
\end{example}

\begin{example}[Sheep between two dogs, asymmetric case]
\label{ex_2dogs_ab}
Assume that $a_1=a$ and $b_{N+1}=b$ with $0<a\leq b<1$, 
and $b_1=a_2=b_2= \cdots =a_{N+1}=1$;
this is a generalization of Example~\ref{ex_2dogs_symm}, in which 
the drift of the left dog is permitted to be stronger than the right.
Here, \eqref{eq:rho-stable} and \eqref{eq:v-stable}
imply that $\hv_{N+1} =\frac{b-a}{N+1} \geq 0$ 
and $\rho_k = a+\frac{(b-a)k}{N+1}$.
\end{example}

	\begin{example}[Recurrence of small systems with constant drifts for $N \leq 2$]
	\label{ex:all-equal2}
	Consider again the setting of Example~\ref{ex:all-equal}, where $b_i - a_i \equiv u \in \R$ for all $i \in [N+1]$.
	If $N=1$ (two particles), then the argument in Example~\ref{ex:all-equal} shows that the system is recurrent, in the sense that $\liminf_{t \to \infty} ( X_{2} (t) - X_1 (t) ) = 1$, a.s.
	In the case $N=2$ (three particles), it can be shown that the system is also recurrent, i.e.,  $\liminf_{t \to \infty} ( X_{3} (t) - X_1 (t) ) = 2$, a.s.;
	it is, however, `critically recurrent' in that excursion durations are heavy-tailed.
	This follows by considering $(\eta_1(t), \eta_2 (t))_{t \in \RP}$ as a random walk with boundary reflection in the quarter-plane $\ZP^2$
	and applying results of Asymont \emph{et al.}~\cite{afm}. To apply these results, it is most convenient to work with the corresponding discrete-time embedded process. Then, in the notation of~\cite[p.~942]{afm}, the entries in the interior covariance matrix are
	\[ \lambda_x = \frac{a_1+b_1+a_2+b_2}{r}, ~ \lambda_y = \frac{a_2+b_2+a_3+b_3}{r}, ~ R = -\frac{a_2+b_2}{r} ,\]
	where $r := \sum_{i=1}^3 (a_i + b_i)$, while the reflection vector components are $M'_x = - (u+a_2)/r'$, $M'_y = (a_2+b_3)/r'$,
	$M''_x = (a_1 +b_2)/r''$, $M''_y = (u - b_2)/r''$, where $r' := a_1 + b_1 + a_2 + b_3$ and $r'' := a_1 + b_2 + a_3 + b_3$.
	Then~\cite[Thm.~2.1]{afm} shows that one has recurrence if
	\[ - \lambda_x \frac{M''_y}{M''_x} - \lambda_y \frac{M'_x}{M'_y} + 2 R \geq 0 ,\]
	which in this case amounts, after a little simplification, to
	\[ \frac{(b_2 -u) (b_1 + a_2)}{a_1 +b_2} + \frac{(a_2+u) (b_2+a_3)}{a_2 + b_3}  \geq a_2 + b_2 ,\]
	which is true, with equality, since $b_1 + a_2 = a_1 + b_2$ and $b_2 + a_3 = a_2 + b_3$.
	
	On the other hand, if $N =3$, we believe that $\lim_{t \to \infty} ( X_{4} (t) - X_1 (t) ) = \infty$, a.s.;
	when $b_i \equiv a_i \equiv 1$ for all $i \in [4]$, it seems possible to approach this by reducing the problem
	to a random walk on $\Z^3$ that is zero drift everywhere, and homogeneous unless one or more coordinates is zero. 
	We conjecture that the same is true
 for any four (or more) consecutive particles in a constant-drift system with $N \geq 4$, but, since such questions appear rather delicate and 
	are not directly relevant for the main phenomena of the present paper, we do not explore them further here.
		\end{example}

	It is of interest to extend the questions posed in Example~\ref{ex:all-equal2} to the setting where
	there are several non-singleton stable clouds with the same speed. For example, if $\Theta = (\theta_1, \theta_2)$
	consists of two stable clouds with $\hv (\theta_1) = \hv (\theta_2)$, then we believe strongly that the system is recurrent, i.e., $\liminf_{t \to \infty} ( X_{N+1} (t) - X_1 (t)) < \infty$, a.s.
	A proof of this might be built on the central limit theorem, Theorem~\ref{thm:clt}, as follows. Suppose the contrary, that the system is transient, i.e., the two stable clouds drift apart. Then the clouds evolve essentially independently,
	so each should satisfy Theorem~\ref{thm:clt} with the same speed and variance on the same scale (with different constants), which suggests they would be in the wrong order with positive probability, providing a contradiction.
The following problem deals with the case of more than two clouds, and settling it seems harder.
	
		\begin{conjecture}
		Suppose that all clouds have the same speed, i.e., $\hv (\theta)$ is constant for $\theta \in \Theta$.
		We expect that, if $\vert\Theta\vert = 3$, then $\liminf_{t \to \infty} ( X_{N+1} (t) - X_1 (t)) < \infty$, a.s.
On the other hand, we expect that, if $\vert\Theta\vert \geq 4$, then $\liminf_{t \to \infty} ( X_{N+1} (t) - X_1 (t)) = \infty$, a.s.
		\end{conjecture}

		 The
		mathematical development to prove the results stated above begins  in  Section~\ref{sec:jackson}, where we
		explain the connection between our model and an appropriate Jackson queueing network, and build on
		the classical work of Goodman and Massey~\cite{GM84} to give some fundamental results on stability and dynamics.
	 Section~\ref{sec:proofs} develops these ideas further to examine the structure underlying Algorithm~\ref{alg:partition},
	 and here we give the proofs of Theorems~\ref{thm:main} and~\ref{thm:algorithm}, and their corollaries.
Section~\ref{sec:clt} is devoted to the central limit theorem in the stable case, and presents the proofs
	 of Theorems~\ref{thm:clt} and~\ref{thm:two-particle-clt}.
			 Finally, in Section~\ref{sec:discussion} we discuss the relationship between the lattice model that we study here and 
		a family of continuum models that have been studied in the literature, and mention some open problems in that direction.
		
\section{Representation as a Jackson network}
\label{sec:jackson}

Consider a system of $N \in \N$ queues, labelled by $[N]$. The parameters of the system are arrival rates $\lambda = (\lambda_i)_{i \in [N]}$,
service rates $\mu = (\mu_i)_{i \in [N]}$ and $P = (P_{ij})_{i,j\in[N]}$, a sub-stochastic routing matrix.
Exogenous customers entering the system arrive at queue $i \in [N]$ via an independent Poisson process of rate $\lambda_i \in \RP$.
Queue $i \in [N]$ serves customers at exponential rate $\mu_i \in \RP$. Once a customer at queue $i$ is served, it is routed to a queue~$j$
with probability $P_{ij}$, while with probability $Q_i := 1 - \sum_{j \in [N]} P_{ij}$ the customer departs from the system.

Provided $\sum_{i \in [N]} \lambda_i >0$ and $\sum_{i \in [N]} Q_i >0$, customers both enter and leave the system, and
it is called an \emph{open Jackson network}. We assume that every queue can be \emph{filled}, 
meaning that, for every $i \in[N]$, there is a $j \in [N]$ and $k \in \ZP$ for which $\lambda_j > 0$ and $(P^k)_{ji}>0$,
and that every queue can be \emph{drained}, meaning that,  for every $i \in[N]$, there is a $j \in [N]$ and $k \in \ZP$ for which $Q_j > 0$ and $(P^k)_{ij}>0$.
The process that tracks the number of customers in each queue at time $t \in \RP$ is a continuous-time Markov chain on $\ZP^N$.
Jackson networks are named for early contributions of R.R.P.~Jackson~\cite{jackson54} and J.R.~Jackson~\cite{jackson57,jackson63}. For a general overview see e.g.~\cite[Ch.~2~\&~7]{ChenYao} or~\cite[Ch.~1]{serfozo}.

\begin{figure}[t]
\begin{center}
\includegraphics{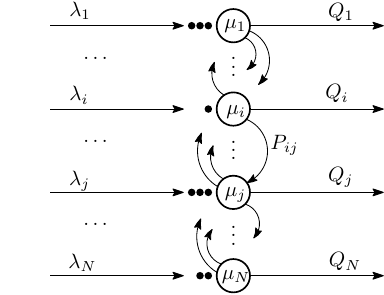}\quad\,
\smallskip

\caption{Schematic of a Jackson network on~$N$ nodes. The $\lambda_i$ are the exogenous arrival rates, the $\mu_i$ are the service rates, and $P$ is the (substochastic) routing matrix by which customers are
redirected following service. From queue~$i$, particles depart the system at rate $Q_i = 1 - \sum_{j} P_{ij}$. The process $\eta$ from~\eqref{eq:eta-def} for the particle system
described in Section~\ref{sec:intro} can be interpreted as the queue-length process for a Jackson network with parameters given by~\eqref{eq:mu-def}, \eqref{eq:lambda-def} and~\eqref{eq:P-def}.}
\label{f_JN}
\end{center}
\end{figure}

For $i \in [N]$ recall from~\eqref{eq:eta-def} that
$\eta_i (t) = X_{i+1} (t) - X_i (t) -1$,
the number of empty sites between consecutive particles at time $t\in\RP$. 
We claim that the process $\eta = (\eta_{i})_{i \in [N]}$ is  precisely the queue-length process for a corresponding Jackson network, namely,
the Jackson network with parameters $\lambda, \mu$ and $P$ given as functions of $(a_i,b_i)_{i \in [N+1]}$ through formulas~\eqref{eq:mu-def}, \eqref{eq:lambda-def} and~\eqref{eq:P-def}.

To see this, observe that `customers' in the queueing network correspond to unoccupied sites between particles in the particle system. 
Exogenous customers enter the network only when the leftmost particle jumps to the
left (rate $a_1$) or the rightmost particle jumps to the right (rate $b_{N+1}$). Customers at queue $i$ are `served' if either the particle at the left end of the interval jumps right
(rate $b_i$) or if the particle at the right end of the interval jumps left (rate $a_{i+1}$). If the 
particle at the left end of the interval jumps right, the customer 
is routed to queue $i-1$ (if $i \geq 2$, at rate $b_i = \mu_i P_{i,i-1}$) or leaves the system. Similarly in the other case.
Customers leave the system only when the leftmost particle jumps to the right (rate $\mu_1 Q_1 = b_1$)
or the rightmost particle jumps to the left (rate $\mu_{N} Q_N = a_{N+1}$).

For some of the results in this section, we will relax assumption~\eqref{ass:positive-rates} to the following;
but see the remark below~\eqref{ass:positive-rates} about generality.

\medskip
\begin{description}
\item\namedlabel{ass:fill-empty}{$\text{A}'$} 
Suppose that either $a_i >0$ for all~$i \in [N+1]$, or $b_i >0$ for all~$i \in [N+1]$.
\end{description}
\medskip

If~\eqref{ass:fill-empty} holds, then the parameters $\lambda, \mu$ and $P$ satisfy the conditions to ensure that
the system~$\eta$ is an open Jackson network in which every queue can be filled and drained:
for example, if $b_i >0$ for all~$i$, then  customers can enter the network at queue~$N$ (since $\lambda_N > 0$), progress in sequence through queues $N-1, N-2, \ldots$, and exit through queue~$1$ (since $Q_1 = \frac{b_1}{b_1+a_2} >0$).
The following statement summarizes the Jackson representation.

\begin{proposition}
\label{prop:Jackson-representation}
		Let $N \in \N$ and suppose that~\eqref{ass:fill-empty} holds. 
The evolution
of the process $\eta = (\eta_{i})_{i \in [N]}$
can be described by an open Jackson network in which every queue can be filled and drained,
with $\eta_i(t)$ being the 
number of customers at queue $i \in [N]$ at time~$t\in\RP$. The  exogenous 
arrival rates $\lambda$ are given by~\eqref{eq:lambda-def},
service rates $\mu$ are given by~\eqref{eq:mu-def},
and routing matrix $P$ is given by~\eqref{eq:P-def}.
\end{proposition}

\begin{remark}
\label{rem:jackson-exclusion}
We believe that Proposition~\ref{prop:Jackson-representation} is known, but we could not find the statement 
explicitly in the literature, and are not confident of an attribution. 
For simple exclusion with homogeneous jump rates one can find the result expressed by Kipnis~\cite[pp.~398--9]{kipnis}.
Spitzer~\cite[pp.~280--1]{spitzer} 
shows that the gaps of empty sites in exclusion process viewed from a tagged particle have
 a product-geometric stationary distribution,
but queues are not mentioned;
 the closest we have been able to find in~\cite{spitzer}
to an identification of exclusion and Jackson dynamics (rather than just stationary distributions)
 is Spitzer's report of
an observation by Kesten~\cite[p.~281]{spitzer} that Poisson streams can be identified in exclusion processes.
\end{remark}

The Jackson representation
given by Proposition~\ref{prop:Jackson-representation} enables us to obtain the following result,
which is, to a substantial degree, a translation to our context of a classical 
result for Jackson networks due to Goodman and Massey~\cite{GM84}, which characterizes
the stable subset $S$ of the queues in the network:
when $S = [N]$ the network is stable, when $S= \emptyset$ it is unstable,
and otherwise it is partially stable~(cf.~\cite{afsw}).

\begin{proposition}
\label{prop:general-Jackson}
		Let $N \in \N$ and suppose that~\eqref{ass:fill-empty} holds. 
The general traffic equation~\eqref{eq:general-balance} has a unique solution, which we denote by $\nu = (\nu_i)_{i \in [N]}$.
Write $\rho_i := \nu_i / \mu_i$ for all $i \in [N]$ and $S := \{ i \in [N] : \rho_i < 1 \}$.
For every $z = (z_i)_{i \in S}$ with $z_i \in \ZP$, we have
\begin{align}
\label{eq:general-Jackson-limit}
 \lim_{t \to \infty} \frac{1}{t} \int_0^t \ind \left( \bigcap_{i \in S} \left\{ \eta_i (s) = z_i \right\} \right) \ud s 
& =  \lim_{t \to \infty} \IP \left[ \bigcap_{i \in S} \left\{ \eta_i (t) = z_i \right\}\right]  \nonumber\\
& = \prod_{i \in S} (1-\rho_i) \rho_i^{z_i} ,\end{align}
where the first equality holds a.s.~and in $L^1$. Moreover, for each fixed $\eta(0) \in \ZP^N$, there exists $\delta >0$ such that, 
\begin{equation}
\label{eq:eta-bound}
\sup_{t \in \RP} \IP \left( \eta_i (t) \geq z \right) \leq \re^{-\delta z} , \text{ for all } i \in S \text{ and all } z \in \RP.
\end{equation}
On the other hand, for $U := [N] \setminus S = \{ i \in [N] : \rho_i \geq 1\}$ and all $B < \infty$,
\begin{equation}
\label{eq:mass-escape} 
\lim_{t \to \infty} \frac{1}{t} \int_0^t \ind \left( \bigcup_{i \in U} \left\{ \eta_i (s) \leq B\right\} \right) \ud s
= 
\lim_{t \to \infty} \IP \left[ \bigcup_{i \in U} \left\{ \eta_i (t) \leq B\right\} \right] = 0 ,\end{equation}
where the first equality holds a.s.~and in $L^1$. 
\end{proposition}
\begin{proof}
Existence and uniqueness of the solution to~\eqref{eq:general-balance}
under the stated conditions follows from Thm.~1 of Goodman and Massey~\cite{GM84},
which also gives the rightmost equalities in~\eqref{eq:general-Jackson-limit} and~\eqref{eq:mass-escape}.

The convergence of the time-averages in~\eqref{eq:general-Jackson-limit} and~\eqref{eq:mass-escape}
 we deduce from the construction of Goodman and Massey.  For a stable system, the Markov chain ergodic theorem 
applies directly; in the general case $U \neq \emptyset$, the 
 idea is to construct modified processes (by adjusting relevant Jackson network parameters) that are stable and satisfy
appropriate stochastic comparison inequalities with components of $\eta$, and again appeal to the ergodic theorem. For convenience, we recapitulate the main steps in the construction here,
pointing to~\cite{GM84} for some more details; we start with the upper bound.

Following~\cite[p.~865--6]{GM84}, we construct a continuous-time Markov chain $\eta^+$ with components $\eta^+_i (t)$, $t \in \RP$, $i \in [N]$,
by re-routing certain customers. Specifically, for every $i \in U$ and $j \in S$ with $P_{ij} > 0$, we (i)~declare that customers, instead of flowing from $i$ to $j$, will depart the system, and (ii)~add an exogenous arrival stream into~$j$
with rate equal to $\mu_i P_{ij}$, the maximal flow rate from $i$ to $j$. In other words, the modified system $\eta^+$ is a Jackson network
with parameters $\mu^+_i = \mu_i$, 
\[ \lambda^+_j := \lambda_j + \1 { j \in S } \sum_{i \in U} \mu_i P_{ij}, \text{ and } P^+_{ij} := \1 { i \notin U \text{ or } j \notin S } P_{ij} .\]
Since no customers flow from $U$ to $S$ in $\eta^+$, the process $\eta^+_S := ( \eta^+_i )_{i \in S}$ observed only on $S$ is itself
the queue-length process of a Jackson network, with parameters $\lambda^+_j = \lambda_j + \sum_{i \in U} \mu_i P_{ij}$, $j \in S$, and $P^+_{ij} = P_{ij}$, $i,j \in S$.
The solution to the general traffic equations for $\eta^+_S$ coincides with the solution $\nu$ to the equations for $\eta$ restricted to~$S$,
and hence, by definition of~$S$, $\eta^+_S$ is stable. Then we apply
the ergodic theorem for 
	 irreducible, continuous-time Markov chains (see e.g.~Thm.~3.8.1 of~\cite[p.~126]{norris}) to obtain
	\[ \lim_{t \to \infty} \frac{1}{t} \int_0^t \ind \left( \bigcap_{i \in S} \left\{ \eta^+_i (s) = z_i \right\} \right) \ud s = \prod_{i \in S} (1-\rho_i) \rho_i^{z_i} , \as \]
	Moreover, we can construct $\eta$ and $\eta^+$ on the same probability space, so
	that $\eta_i (t) \leq \eta^+_i (t)$ for all $i \in [N]$ and all $t \in \RP$: see~\cite[p.~867--8]{GM84} for details.
	Hence we conclude that
\begin{equation}
\label{eq:ergodic-limsup}
\limsup_{t \to \infty} \frac{1}{t} \int_0^t \ind \left( \bigcap_{i \in S} \left\{ \eta_i (s) = z_i \right\} \right) \ud s \leq \prod_{i \in S} (1-\rho_i) \rho_i^{z_i} , \as \end{equation}
In the other direction, it is shown in~\cite[pp.~866--7]{GM84}, that for any $\eps >0$,
by increasing the rate of departures from the system for every node in~$U$, one can construct
a process $\eta^\eps$ which is stable, for which $\eta_i (t) \geq \eta^\eps_i (t)$ for all $i \in [N]$
and all $t \in \RP$.
 In more detail,
if $\nu = (\nu_i)_{i \in [N]}$ is the solution to~\eqref{eq:general-balance} for $\eta$, then $\eta^\eps$ is a Jackson system with parameters 
$\lambda^\eps, \mu^\eps, P^\eps$, where, for $\eps \geq 0$,
$\lambda^\eps_i := \lambda_i$, 
\begin{equation}
\label{eq:parameters-lower-bound}
\mu^\eps_i := \begin{cases} \mu_i & \text{if } i \in S, \\ \nu_i + \eps & \text{if } i \in U, \end{cases} ~\text{and}~ P^\eps_{ij} :=  \begin{cases} P_{ij} & \text{if } i \in S, \\ \frac{\mu_i}{\nu_i + \eps} P_{ij} & \text{if } i \in U. \end{cases} \end{equation}
For $i \in U$, we have $\nu_i \geq \mu_i$, and hence $Q^\eps_i := 1 - \sum_{j \in [N]} P^\eps_{ij} > 0$ for $\eps>0$, 
so customers served at states in~$U$ may now depart the system.
Note that, by~\eqref{eq:parameters-lower-bound}, $\lim_{\eps \to 0} \mu^\eps = \mu^0$
and $\lim_{\eps \to 0} P^\eps = P^0$. We also note that, for every $\eps \geq 0$, $P^\eps$ is an irreducible substochastic matrix with at least one row sum strictly less than~$1$,
so the matrix power $(P^\eps)^k$ tends to $0$ as $k \to \infty$.

Let $\nu^\eps$ be the solution of the general traffic equation~\eqref{eq:general-balance} for $\eta^\eps$.
As shown in~\cite[p.~867]{GM84}, one also has $\lim_{\eps\to 0} \nu^\eps = \nu^0$ where $\nu^0$ satisfies $\nu^0 = (\nu^0 \wedge \mu^0) P^0 + \lambda$;
we will show that
$\nu^\eps \leq \nu^0 = \nu$.
 By construction,
it is the case that $\nu = (\nu \wedge \mu) P + \lambda = \nu P^0 + \lambda$, and since $P^\eps_{ij} \leq P^0_{ij}$ for all $i,j$, we get
$\nu^\eps \leq \nu^\eps P^0 + \lambda$. 
Hence $\tilde\nu:= \nu^\eps - \nu$ satisfies $\tilde \nu \leq P^0 \tilde \nu$, and since $(P^0)^k$ tends to $0$ as $k \to \infty$,
it follows that $\nu^\eps_i \leq \nu_i$ for all $i \in [N]$. Moreover,
since $\nu_i < \mu_i^\eps$ for all $i \in [N]$ and every $\eps>0$, by definition of $\mu^\eps$, this means that $\nu^\eps_i < \mu_i^\eps$ for all $i \in [N]$. Hence $\eta^\eps$ is stable for every $\eps >0$.
Thus in fact $\nu^\eps = \nu^\eps P^\eps + \lambda$, i.e., $\nu^\eps (I - P^\eps) = \lambda$, for $\eps >0$.
For every $\eps \geq 0$, the matrix $I-P^\eps$ is an \emph{M-matrix}, and hence invertible 
(this follows from e.g.~Lemma~7.1 of~\cite{ChenYao} and the fact that $(P^\eps)^k$ tends to $0$ as $k \to \infty$).
Therefore $\nu^\eps = \lambda (I-P^\eps)^{-1}$  for $\eps>0$, and $\nu^0 = \lim_{\eps \to 0} \nu^\eps = \lambda (I - P^0)^{-1} = \nu$.
In particular, from~\eqref{eq:parameters-lower-bound}, we have that for $i \in S$, $\lim_{\eps \to 0} \rho^\eps_i = \lim_{\eps \to 0} \nu^\eps_i / \mu_i = \rho_i <1$,
while for $i \in U$, $\lim_{\eps \to 0} \rho^\eps_i = \lim_{\eps \to 0} \nu^\eps_i / \mu^\eps_i = 1$.
Hence, by another application
of the ergodic theorem,
\begin{equation}
\label{eq:ergodic-liminf}
\liminf_{t \to \infty} \frac{1}{t} \int_0^t \ind \left( \bigcap_{i \in S} \left\{ \eta_i (s) = z_i \right\} \right) \ud s \geq \prod_{i \in S} (1-\rho^\eps_i) (\rho_i^\eps)^{z_i} , \as \end{equation}
	Taking $\eps \to 0$ in~\eqref{eq:ergodic-liminf} and combining with~\eqref{eq:ergodic-limsup}, we obtain the a.s.~convergence in~\eqref{eq:general-Jackson-limit}; the $L^1$ convergence follows from the bounded
	convergence theorem. 	On the other hand, if $i \in U$, then, for any $z \in \N$, 
	\[ \limsup_{t \to \infty} \frac{1}{t} \int_0^t \1 { \eta_i (s) \leq z } \ud s \leq  
	1 - \liminf_{t \to \infty} \frac{1}{t} \int_0^t \sum_{k=1}^\infty \1 { \eta_i (s) = z+k } \ud s .\]
	Here, by Fatou's lemma, a.s.,
	\begin{align*}
	\liminf_{t \to \infty} \frac{1}{t} \int_0^t \sum_{k=1}^\infty \1 { \eta_i (s) = z+k } \ud s  & \geq  \sum_{k=1}^\infty 	\liminf_{t \to \infty} \frac{1}{t} \int_0^t  \1 { \eta_i (s) = z+k } \ud s \\
	& \geq  \sum_{k=1}^\infty 	 ( 1 - \rho_i^\eps ) ( \rho_i^\eps )^{z+k} = ( \rho_i^\eps )^{z+1} , \end{align*}
	and then taking $\eps \to 0$ gives the a.s.~convergence~\eqref{eq:mass-escape}; again, bounded convergence gives the $L^1$~case.
	
	It remains to prove~\eqref{eq:eta-bound}. Fix $\eta(0) \in \ZP^N$. 
The Goodman and Massey construction furnishes the coupling of $\eta$ and $\eta^+$ such that~$\eta_i(t) \leq \eta_i^+(t)$ for all $i \in [N]$ and all $t \in \RP$, where $\eta^+_S$ is an ergodic Jackson network over $S$, i.e., $\max_{i \in S} \rho^+_i < 1$.
Hence, by~\eqref{eq:general-Jackson-limit},   there exists $\delta >0$ such that $\lim_{t\to \infty} \IP ( \eta_i (t) \geq z )  \leq \lim_{t\to \infty} \IP ( \eta^+_i (t) \geq z )  \leq \re^{-\delta z}$ for all $i \in S$
	and all $z \in \RP$. 
	Moreover, since a stable Jackson network is exponentially ergodic,
	as is established in Thm.~2.1 of~\cite{fmms2} via a Lyapunov function approach (see also~\cite{fmms1}), or
	in~\cite{ls} via a spectral gap approach, it follows that there exists $\delta >0$ such that 
	$\IP ( \eta_i (t) \geq z ) \leq \IP ( \eta^+_i (t) \geq z ) \leq \re^{-\delta z} + \re^{-\delta t}$ for all $i \in S$ and all $t, z \in \RP$.
	Thus we can conclude that, for any $\eps >0$ there exists $\delta >0$ such that
	$\sup_{t \geq \eps z} \IP ( \eta_i (t) \geq z ) \leq \re^{-\delta z}$ for all $z \in \RP$ and all $i \in S$.
In addition, if $\zeta_t$ denotes the total number of exogenous arrivals in the Jackson network $\eta$ over time~$[0,t]$,
then $\max_{i \in [N]} ( \eta_i (t) - \eta_i (0) ) \leq \zeta_t$, a.s., and $\zeta_t$ is Poisson with mean $Ct$
for $C := \sum_{i \in [N]} \lambda_i < \infty$. Hence 
 $\IP ( \eta_i (t) - \eta_i (0) \geq  z ) \leq \IP ( \zeta_t \geq z )$,
so, for $\eps >0$ small enough, 
$\sup_{0 \leq t \leq \eps z}  \IP ( \eta_i (t) - \eta_i (0) \geq z ) \leq \IP ( \zeta_{\eps z} \geq z )$, and this decays exponentially in~$z$, by standard Poisson tail bounds.
Combining these bounds yields~\eqref{eq:eta-bound}.
\end{proof}

Using Proposition~\ref{prop:general-Jackson}, we show in Theorem~\ref{thm:speeds} that each particle
 satisfies a law of large numbers with a  deterministic \emph{asymptotic speed}.
Before stating that result, we state an algebraic result on the quantities that will play the roles of the speeds.

\begin{lemma}
\label{lem:speeds-algebra}
		Let $N \in \N$ and suppose that~\eqref{ass:fill-empty} holds. 
Suppose that $\nu = (\nu_i)_{i \in [N]}$ is the unique solution to the
general traffic equation~\eqref{eq:general-balance},
whose existence is guaranteed by Proposition~\ref{prop:general-Jackson}.
 Write $\rho_i := \nu_i / \mu_i$, and adopt the convention  that $\rho_0 = \rho_{N+1} = 1$.
Define
\begin{align}
\label{eq:speed-per-particle}
v_i & := \left( 1 \wedge \rho_i \right) b_i - \left( 1 \wedge \rho_{i-1} \right) a_i , \text{ for all } i \in [N+1].\end{align}
Then, writing $x^+ := x \1 { x \geq 0}$ for $x \in \R$, we have
\begin{align}
\label{eq:speed-neighbours}
v_{i+1} - v_i & =  (\rho_i -1)^+ ( b_i + a_{i+1} ) , \text{ for all } i \in [N] .
\end{align}
In particular, if $i \in [N]$ is such that $\rho_i \leq 1$, then $v_{i+1} = v_{i}$, while if $\rho_i >1$, then $v_{i+1} > v_i$.
\end{lemma}

Here is our result on existence of speeds that underlies much of our analysis.

\begin{theorem}
\label{thm:speeds}
		Let $N \in \N$ and suppose that~\eqref{ass:fill-empty} holds. 
	Then there exist $- \infty < v_1 \leq \cdots \leq v_{N+1} < \infty$ such that
\begin{equation}
\label{eq:speeds-exist}
 \lim_{t \to \infty} \frac{X_i (t)}{t} = v_i, \as , \text{ for every } i \in [N+1]. \end{equation}
Moreover,
if $\nu = (\nu_i)_{i \in [N]}$ is the unique solution to the
general traffic equation~\eqref{eq:general-balance}, and $\rho_i := \nu_i / \mu_i$, 
then the $v_i$ in~\eqref{eq:speeds-exist} are the quantities defined by~\eqref{eq:speed-per-particle}.
 \end{theorem}

We first prove the lemma.

\begin{proof}[Proof of Lemma~\ref{lem:speeds-algebra}.]
Suppose that $(\nu_i)_{i \in [N]}$ solves $\nu = ( \nu \wedge \mu ) P + \lambda$.
Write $\rho_i := \nu_i / \mu_i$, and take $\rho_0 = \rho_{N+1} = 1$.
Using the expressions~\eqref{eq:mu-def}, \eqref{eq:lambda-def}, and~\eqref{eq:P-def}, the general traffic equation~\eqref{eq:general-balance} can then be expressed in terms of the $\rho_i$, $a_i$, and $b_i$ as 
\begin{equation}
\label{eq:balance-rho}
 (b_i + a_{i+1} ) \rho_i = (1 \wedge \rho_{i-1} ) a_i + (1 \wedge \rho_{i+1} ) b_{i+1} , \text{ for all } i \in [N] .\end{equation}
Using~\eqref{eq:speed-per-particle} twice in~\eqref{eq:balance-rho}, we obtain, for $i \in [N]$, 
\[ (b_i + a_{i+1} ) \rho_i = (1 \wedge \rho_{i-1} ) a_i + v_{i+1} + ( 1 \wedge \rho_i ) a_{i+1}
= (1 \wedge \rho_i ) b_i  + ( 1 \wedge \rho_i ) a_{i+1} + v_{i+1} - v_i .\]
This yields~\eqref{eq:speed-neighbours}, and hence the final sentence of the lemma.
\end{proof}

\begin{proof}[Proof of Theorem~\ref{thm:speeds}.]
Suppose that $(\nu_i)_{i \in [N]}$ solves $\nu = ( \nu \wedge \mu ) P + \lambda$.
Write $\rho_i := \nu_i / \mu_i$. 
Fix $i \in [N+1]$. Let $N^-_i (t)$ and $N^+_i (t)$, $t \in \RP$, be two independent homogeneous Poisson processes, of rates $a_i$ and $b_i$, respectively.
For all $t \in \RP$, set $\eta_0 (t) := +\infty$ and $\eta_{N+1} (t) := +\infty$.
Then we have the stochastic integral representation
\begin{equation}
\label{eq:poisson-calculus}
 X_i (t) - X_i (0) = \int_0^t \1 { \eta_{i} (s-) \geq 1 } \ud N^+_i (s) -  \int_0^t \1 { \eta_{i-1} (s-) \geq 1 } \ud N^-_i (s) ,\end{equation}
since the attempted jumps are suppressed if the corresponding `queues' are empty. Here, as usual $\eta_i (s- ) := \lim_{u \uparrow s} \eta_i (u)$ for $s >0$, and
$\eta_i (0-) := \eta_i (0)$.

Write $M_i^+ (t) := N_i^+ (t) - b_i t$ and $M_i^-(t) := N_i^- (t) - a_i t$, $t \in \RP$.
Then $M_i^\pm$ are continuous-time, square-integrable martingales, and we can re-write~\eqref{eq:poisson-calculus} as
\begin{align}
\label{eq:x-representation}
  X_i (t) - X_i (0) & = b_i \int_0^t \1 { \eta_{i} (s-) \geq 1 } \ud s -  a_i \int_0^t \1 { \eta_{i-1} (s-) \geq 1 } \ud s  \nonumber\\
	& {}  + \int_0^t \1 { \eta_{i} (s-) \geq 1 } \ud M^+_i (s)  -  \int_0^t \1 { \eta_{i-1} (s-) \geq 1 } \ud M^-_i (s) .\end{align}
	Recall from the statement of Proposition~\ref{prop:general-Jackson} that $S$ denotes the set of $i \in [N]$ such that $\rho_i < 1$. 
	Then we have from Proposition~\ref{prop:general-Jackson} 
	that, for $i \in [N]$, a.s.,
	\begin{align}
	\label{eq:blocking-ergodicity}
	\lim_{t\to\infty} \frac{1}{t} \int_0^t \1 { \eta_{i} (s-) \geq 1 } \ud s & = 1 - 	\lim_{t\to\infty} \frac{1}{t} \int_0^t \1 { \eta_{i} (s) = 0  } \ud s \nonumber\\
& = 	1 - (1- \rho_{i} ) \1 { i \in S} = 1 \wedge \rho_i.
	\end{align}
	Note that~\eqref{eq:blocking-ergodicity} also holds, trivially, for $i = 0, N+1$.
	Let $H^+_i (t) := \1 { \eta_{i} (t-) \geq 1 } = 1 \wedge \eta_i (t-)$. Then $H_i^+$ is left continuous with right limits, and hence~\cite[p.~63]{protter} 
	$Y_i^+ := H_i^+ \cdot M^+_i$ given by $Y_i^+ (t) := \int_0^t H_i^+ (s) \ud M^+_i(s)$ is  a right-continuous local martingale. Its quadratic
	variation process $[ Y_i^+]$ satisfies
	\[ [ Y_i^+  ]_t = \int_0^t (H^+_i (s))^2 \ud [ M^+_i ]_s = \int_0^t H^+_i (s)  \ud N^+_i (s) \leq N^+_i (t), \text{ for } t \in \RP, \]
using~\cite[pp.~75--76]{protter}, with the fact that $[ M^+_i]_t = [N^+_i]_t = N^+_i (t)$ is the quadratic variation of the (compensated) Poisson process~\cite[p.~71]{protter}.
Hence $\IE [ (Y^+_i (t))^2 ] = \IE ( [ Y^+_i]_t ) \leq \IE N^+_i (t) = b_i t$, so $(Y_i^+)^2$ is a non-negative, right-continuous submartingale.
By an appropriate maximal inequality~\cite[p.~13]{ks}, for any $p >1$,
\begin{equation}
\label{eq:doob-inequality}
 \IP \left[ \sup_{0 \leq s \leq t} \vert { Y^+_i (s) } \vert \geq t^{p/2} \right] \leq t^{-p} \IE [ (Y_i^+ (t))^2 ]  \leq b_i t^{1-p} , \text{ for all }   t \in (0, \infty). \end{equation}
Applying the Borel--Cantelli lemma with~\eqref{eq:doob-inequality} along subsequence $t = t_n := 2^n$, it follows that, a.s., for all but finitely many $n \in \ZP$,
$\sup_{0 \leq s \leq 2^n} \vert { Y^+_i (s) } \vert \leq 2^{np/2}$. Every $t \geq 1$ has $2^n \leq t < 2^{n+1}$ for $n = n(t) \in \ZP$, and so, a.s., for all $t \in \RP$ sufficiently large,
\[ \sup_{0 \leq s \leq t} \vert { Y^+_i (s) } \vert \leq \sup_{0 \leq s \leq 2^{n+1}} \vert { Y^+_i (s) } \vert \leq 2^{(n+1)p/2} \leq 2^{p/2} \cdot t^{p/2} .\]
It follows that, for any $p' > p >1$, $t^{-p'/2} \vert { Y^+_i (t) } \vert \to 0$, a.s., as $t \to \infty$.
Together with the analogous argument involving $M_i^-$, we have thus shown that, a.s., 
	\begin{equation}
		\label{eq:martingale-control}
		\lim_{t\to\infty} \frac{1}{t} \!\int_0^t\! \1 { \eta_{i} (s-) \geq 1 } \ud M^+_i (s) =  \lim_{t\to\infty} \frac{1}{t} \!\int_0^t\!  \1 { \eta_{i-1} (s-) \geq 1 } \ud M^-_i (s) = 0. \end{equation}
		Combining~\eqref{eq:x-representation} with~\eqref{eq:blocking-ergodicity} and~\eqref{eq:martingale-control},
we conclude that $X_i(t) /t \to v_i$, a.s., where $v_i \in \R$ is given by~\eqref{eq:speed-per-particle}. That
	$v_1 \leq v_2 \leq \cdots \leq v_{N+1}$
 follows from the fact that $X_1(t) < X_2 (t) < \cdots < X_{N+1}(t)$
 for all~$t \in \RP$. This completes the proof of~\eqref{eq:speeds-exist}, with the $v_i$ given by~\eqref{eq:speed-per-particle}, establishing the theorem.
	\end{proof}
	
A consequence of Proposition~\ref{prop:general-Jackson} is that $S = [N]$ (i.e., the system is stable)
if and only if
the (unique) solution $\nu$ to~\eqref{eq:general-balance} satisfies
 $\rho_i := \nu_i / \mu_i < 1$ for every $i \in [N]$. 
Any such $\nu$ thus
 solves also $\nu = \nu P + \lambda$, i.e.,
\begin{equation}
\label{eq:stable-balance}
\nu ( I - P ) = \lambda, 
\end{equation}
where $I$ is the $N$ by $N$ identity matrix. We call~\eqref{eq:stable-balance} the \emph{stable traffic equation};
note that unlike the general traffic equation~\eqref{eq:general-balance}, the system~\eqref{eq:stable-balance} is linear. 
The system~\eqref{eq:stable-balance} in fact classifies whether or not the system is stable, as the following result shows.
Since at this point we invoke the formulas~\eqref{eq:rho-stable} and~\eqref{eq:v-stable}, we now need to assume~\eqref{ass:positive-rates}.

\begin{proposition}
\label{prop:stable-Jackson}
	Let $N \in \N$ and suppose that~\eqref{ass:positive-rates} holds.
The stable traffic equation~\eqref{eq:stable-balance} has a unique solution $\nu = \lambda (I-P)^{-1}$, which we denote by $\nu = (\nu_i)_{i \in [N]}$.
Write $\rho_i := \nu_i / \mu_i$. Then the following hold.
\begin{thmenumi}[label=(\roman*)]
\item 
\label{prop:stable-Jackson-i}
The process~$\eta$ is stable if and only if
$\rho_i < 1$ for all $i \in [N]$. Equivalently, the process is stable if and only if
$\lambda (I -P)^{-1} < \mu$, componentwise.
\item 
\label{prop:stable-Jackson-ii}
The $\rho_i$ are given by the explicit formula~\eqref{eq:rho-stable}, 
where $\hv_{N+1}$ is given by~\eqref{eq:v-stable}.
\item
\label{prop:stable-Jackson-iii}
If $\rho_i <1$ for all $i \in [N]$, then the quantity $\hv_{N+1}$ given by~\eqref{eq:v-stable} specifies the speed of the cloud, via
$\lim_{t\to \infty} t^{-1} X_i (t) = \hv_{N+1}$, a.s., for all $i \in [N+1]$.
\end{thmenumi}
\end{proposition}
\begin{proof}
Part~\ref{prop:stable-Jackson-i} is essentially just a specialization of the characterization of stable Jackson networks (which goes back to Jackson~\cite{jackson57,jackson63}; see e.g.~\cite[\S 2.1]{ChenYao} or~\cite[\S 3.5]{FMM}) to our setting;
we give a proof using Proposition~\ref{prop:general-Jackson}.
Consider the stable traffic equation~\eqref{eq:stable-balance},
where $\lambda$ and $P$ are given by~\eqref{eq:lambda-def} and~\eqref{eq:P-def} respectively.
The matrix $I-P$ is an \emph{M-matrix}, and hence invertible (see Lemma~7.1 of~\cite{ChenYao}) and hence the solution $\nu = \lambda (I-P)^{-1} $
of~\eqref{eq:stable-balance} exists and is unique. 
Proposition~\ref{prop:general-Jackson}
says that if the system is stable, then there is a solution to~\eqref{eq:stable-balance}
with $\nu_i < \mu_i$ for all $i$. Conversely, if there exists a solution to~\eqref{eq:stable-balance}
for which $\nu_i < \mu_i$ for all $i \in [N]$, this solution is necessarily the (unique) solution to the general traffic equation~\eqref{eq:general-balance},
and the system is stable, by Proposition~\ref{prop:general-Jackson}. 
This argument proves part~\ref{prop:stable-Jackson-i}.

In terms of the loads $\rho_i = \nu_i / \mu_i$,   the stable traffic equation~\eqref{eq:stable-balance} reads
\begin{equation}
\label{eq:balance-rho-stable}
 (b_i+a_{i+1}) \rho_i = a_i\rho_{i-1}+b_{i+1} \rho_{i+1}, \text{ for } i \in [N],
\end{equation}
where we impose the boundary condition $\rho_0 = \rho_{N+1} = 1$.
Consider first the system~\eqref{eq:balance-rho-stable} without any boundary condition; then, the solutions $(\rho_0,\rho_1,\ldots,\rho_{N+1})$ to~\eqref{eq:balance-rho-stable} form a linear 
subspace of $\R^{N+2}$, and this subspace is two-dimensional 
because $\rho_0$ and $\rho_1$
 uniquely determine the rest. We identify a basis for this solution space. 
Define vectors 
$\alpha =(\alpha_0, \alpha_1, \ldots, \alpha_{N+1})$
and 
$\beta = (\beta_0,\beta_1, \ldots, \beta_{N+1})$
by $\alpha_0 := 1$, $\beta_0:=0$, and, for $k \in [N+1]$, 
\begin{align*}
 \alpha_k : = \frac{a_1\ldots a_k}{b_1\ldots b_k}, \text{ and }  \beta_k :=\frac{1}{b_k}+\frac{a_k}{b_kb_{k-1}}+   \cdots + \frac{a_k\ldots a_2}{b_k\ldots b_1}.
\end{align*}
Note that in the notation defined at~\eqref{eq:alpha-beta-def}, $\alpha_k = \alpha(1;k)$ and $\beta_k = \beta(1;k)$. 
The  vectors $\alpha,\beta \in \R^{N+2}$ are linearly
independent (because one is strictly positive and the other
is not). Moreover, it is straightforward to check that both $\alpha$ and $\beta$
solve the system~\eqref{eq:balance-rho-stable}: this is familiar from the
usual solution to the difference equations associated with the general gambler's ruin problem~\cite[pp.~106--108]{kt}.
Then, since $\rho_0=1$, 
any solution of~\eqref{eq:balance-rho-stable} with the given boundary conditions must have the form 
$\rho_i = \alpha_i + \hv_{N+1} \beta_i$ for some $\hv_{N+1}\in\R$, which is precisely~\eqref{eq:rho-stable}. Using the condition $\rho_{N+1}=1$, we find that
$\hv_{N+1}$ must be given by~\eqref{eq:v-stable}. This proves part~\ref{prop:stable-Jackson-ii}.

Suppose that $\rho_i < 1$ for every $i \in [N+1]$. Then, by part~\ref{prop:stable-Jackson-i}, the system is stable, and Theorem~\ref{thm:speeds}
shows that $\lim_{t \to \infty} t^{-1} X_i (t) = v_i$, a.s., where, by~\eqref{eq:speed-per-particle}, 
\[ v_1 = \rho_1 b_1 - a_1; ~~ v_i = \rho_i b_i - \rho_{i-1} a_i, \text{ for } 2 \leq i \leq N; ~~ v_{N+1} = b_{N+1} - \rho_N a_N .\]
It follows from~\eqref{eq:speed-neighbours} that $v_1 = v_2 = \cdots = v_{N+1}$; that their common value is $v_1 = \hv_{N+1}$
given by~\eqref{eq:v-stable} 
follows from the $i=1$ case of~\eqref{eq:rho-stable}. This proves~\ref{prop:stable-Jackson-iii}.
\end{proof}

\section{Monotonicity, restriction, and mergers}
\label{sec:proofs}

In this section we study the relationship between the stability characterization of the whole system, in terms of appropriate traffic equations, 
as derived from the Jackson representation and presented in Propositions~\ref{prop:general-Jackson} and~\ref{prop:stable-Jackson} above,
with the traffic equations associated with certain sub-systems of the full system. This will allow us to characterize the maximal stable sub-systems,
and show that their characteristic parameters can be expressed in terms of their intrinsic parameters only,
and hence prove our main stability results, Theorems~\ref{thm:main} and~\ref{thm:algorithm}, and their corollaries presented in Section~\ref{sec:results}.

Let $I = [ \ell;m] \subseteq [N+1]$ be a discrete interval with $m \geq 2$;
recall that $I^\circ = [\ell; m-1]$ excludes the rightmost element.
Define $\lambda^I := (\lambda^I_{i})_{i \in I^\circ}$ by $\lambda^I_{\ell} := a_\ell + b_{\ell+1}$
if $m=2$ ($I^\circ$ is a singleton), and, if $m \geq 3$,
\begin{equation}
\label{eq:lambda-I-def}
\lambda^I_{\ell} := a_\ell, ~ \lambda^I_{\ell+m-2} := b_{\ell+m-1}, \text{ and }
\lambda^I_j := 0 \text{ for } \ell+1 \leq j \leq \ell+m-3.
\end{equation} 
Also, recalling the definition of $\mu_i$ from~\eqref{eq:mu-def}, define the  matrix $(P^I_{i,j})_{i,j \in I^\circ}$ by
\begin{equation}
\label{eq:P-I-def}
\begin{split}
P^I_{i,i-1} & := \frac{b_i}{\mu_i} = \frac{b_i}{b_i+a_{i+1}}, \text{ for } \ell+1 \leq i \leq \ell+m-2; \\
P^I_{i,i+1} & := \frac{a_{i+1}}{\mu_i} = \frac{a_{i+1}}{b_i+a_{i+1}}, \text{ for } \ell \leq i \leq \ell+m-3;
\end{split}
\end{equation}
with $P^I_{ij} := 0$ otherwise. Set $Q^I_i := 1 - \sum_{j \in I^\circ} P^I_{ij}$.
Define $\mu^I := (\mu_i )_{i \in I^\circ}$. 
Note that in the case $I = [N+1]$, $\lambda^{[N+1]} = \lambda$ and $P^{[N+1]} = P$
defined by~\eqref{eq:lambda-I-def} and~\eqref{eq:P-I-def}
coincide with the definitions given previously at~\eqref{eq:lambda-def} and~\eqref{eq:P-def}.
Given $I \subseteq [N+1]$, we call the system
\begin{equation}
\label{eq:reduced-balance}
\nu^I = (\nu^I \wedge \mu^I ) P^I + \lambda^I 
\end{equation}
the \emph{reduced} traffic equation corresponding to~$I$. For a solution  $\nu^I = (\nu^I_i )_{i \in I^\circ}$ to~\eqref{eq:reduced-balance},
we write $\rho^I (i) := \nu^I_i / \mu_i$ for all $i \in I^\circ$.
Then, similarly to~\eqref{eq:balance-rho}, the reduced traffic equation can be written in terms of $(\rho^I_i)_{i \in I^\circ}$ as
\begin{equation}
\label{eq:reduced-traffic-rhos}
( b_i + a_{i+1} ) \rho_i^I = \bigl( 1 \wedge \rho^I_{i-1} \bigr) a_i + \bigl( 1 \wedge \rho^I_{i+1} \bigr) b_{i+1}, \text{ for all } i \in I^\circ , \end{equation}
with the convention that $\rho^I_{\min I -1} = \rho^I_{\max I} = 1$.
Analogously to~\eqref{eq:speed-per-particle}, we then define
\begin{equation}
\label{eq:v-I-def} v^I_i := \bigl( 1 \wedge \rho^I_i  \bigr) b_i - \bigl( 1 \wedge \rho^I_{i-1} \bigr) a_i , \text{ for all } i \in I ,\end{equation}
again with the convention  $\rho^I_{\min I -1} = \rho^I_{\max I} = 1$. Note that $v^I$ satisfies an appropriate version
of the algebraic Lemma~\ref{lem:speeds-algebra}.

The next result concerns solutions of the reduced traffic equation~\eqref{eq:reduced-balance}.

\begin{lemma}
\label{lem:reduced-system}
	Let $N \in \N$ and suppose that~\eqref{ass:positive-rates} holds.
Let $I  \subseteq [N+1]$ be a discrete interval with $\vert I\vert   \geq 2$. 
There is a unique solution $\nu^I$ to~\eqref{eq:reduced-balance};
equivalently, there is a unique solution $\rho^I$ to~\eqref{eq:reduced-traffic-rhos}.
Moreover, if 
$\rho^I (j) \leq 1$ for all $j \in I^\circ$,
then $\rho^I = \hrho_I$ as defined at~\eqref{eq:rho-I-def}
and, for all $i \in I$, $v^I_i = \hv(I)$ with the definitions at~\eqref{eq:v-I-def} and~\eqref{eq:hv-def}.
\end{lemma}
\begin{proof}
Existence and uniqueness of the solution to~\eqref{eq:reduced-balance} follows from the
results of~\cite{GM84}, exactly as 
existence and uniqueness of the solution to the general traffic equation~\eqref{eq:general-balance}
in Proposition~\ref{prop:general-Jackson}.
If $\rho^I (j) \leq 1$ for all $j \in I^\circ$, then~\eqref{eq:reduced-balance}
coincides with the system $\nu^I = \nu^I P^I + \lambda^I$,
a reduced version of the stable traffic equation~\eqref{eq:stable-balance}, 
 and the argument of 
Proposition~\ref{prop:stable-Jackson} implies that $\rho^I$ and $v^I$ satisfy the appropriate versions of~\eqref{eq:rho-stable} and~\eqref{eq:v-stable},
which establishes that $\rho^I = \hrho_I$ as defined at~\eqref{eq:rho-I-def} and $v^I_i = \hv(I)$ as defined at~\eqref{eq:hv-def}.
\end{proof}

Given a discrete interval $I \subseteq [N+1]$, we say $I$ is a \emph{candidate} stable cloud
if the solution to the reduced traffic equation~\eqref{eq:reduced-balance}, or, equivalently, the system~\eqref{eq:reduced-traffic-rhos}, gives $\rho^I_i < 1$ for all $i \in I^\circ$.
To facilitate the proof of Theorem~\ref{thm:algorithm}, verifying the correctness of Algorithm~\ref{alg:partition},
we need to identify when candidate stable clouds are genuine stable clouds. Here the key property is that a candidate stable cloud
is necessarily either a stable cloud, or a subset of a stable cloud; hence we need to test whether a candidate stable cloud is maximal,
or whether it can be extended to a larger candidate stable cloud. The next two lemmas present results in this direction.

Lemma~\ref{lem:reduced-system-consistency} is a consistency result, which shows, in particular,
 that the solution $\nu$ to the general traffic equation~\eqref{eq:general-balance},
when restricted to a stable cloud $\theta \in \Theta$, coincides with $\nu^\theta$, the solution to the reduced
  traffic equation~\eqref{eq:reduced-balance} for $I = \theta$.

\begin{lemma}
\label{lem:reduced-system-consistency}
	Let $N \in \N$ and suppose that~\eqref{ass:positive-rates} holds.
Let $I \subseteq [N+1]$ be a discrete interval with $\vert I\vert \geq 2$,
and let $I_0 \subseteq I$ be such that (i) either $\min I_0 = \min I$ or $\rho^I_{\min I_0 -1} \geq 1$, and (ii)
either $\max I_0 = \max I$ or $\rho^I_{\max I_0} \geq 1$.
Then $\nu^I_i = \nu^{I_0}_i$ for all $i \in I_0^\circ$.
\end{lemma}
\begin{proof}
Consider the solution $\nu^I$ to the reduced traffic equation~\eqref{eq:reduced-balance} over~$I$,
and let $\rho^I_i := \nu_i^I / \mu_i$. 
Since $\nu^I$ solves~\eqref{eq:reduced-balance} for $I$, we have from~\eqref{eq:reduced-traffic-rhos} that
 the $\rho^I_i$, $i \in I_0^\circ$, satisfy
\begin{equation}
\label{eq:rho-I-I0}
 (b_i + a_{i+1} ) \rho^{I}_i = \bigl( 1 \wedge \rho^I_{i-1} \bigr) a_i + \bigl( 1 \wedge \rho^I_{i+1} \bigr) b_{i+1} , \text{ for all } i \in I_0^\circ, \end{equation}
with boundary conditions $\rho^{I}_{\min I_0-1} = 1$ (because either $\min I_0 = \min I$,
in which case $\rho^{I}_{\min I_0-1} = \rho^{I}_{\min I-1} = 1$, the left boundary condition for~\eqref{eq:reduced-traffic-rhos} over~$I$, 
or $\min I_0 > \min I$ and $\rho^I_{\min I_0-1} \geq 1$, so $1 \wedge \rho^I_{\min I_0 - 1} = 1$) and $\rho^I_{\max I_0} = 1$ (analogously).
In other words, the $\rho^I_i$ solve exactly the  system~\eqref{eq:reduced-traffic-rhos}  over $I = I_0$,
which is also the system solved by the $\rho^{I_0}_i$.
By uniqueness of solutions to the reduced traffic equation~\eqref{eq:reduced-balance} (see Lemma~\ref{lem:reduced-system}), it follows that $\nu^I_i = \nu^{I_0}_i$ for all $i \in I_0^\circ$. 
\end{proof}

Lemma~\ref{lem:stable-extension} gives conditions when a candidate stable cloud is
further stabilized by the external system, and will allow us to extend a candidate stable cloud either to the left or to the right.
The proof uses similar restriction ideas to the proof of Lemma~\ref{lem:reduced-system-consistency}.

\begin{lemma}
\label{lem:stable-extension}
	Let $N \in \N$ and suppose that~\eqref{ass:positive-rates} holds.
	Let $I \subseteq [N+1]$ be a discrete interval with $\vert I\vert \geq 2$,
and let $I_0 \subseteq I$ be such that $\rho^{I_0}_i \leq 1$ for all $i \in I_0^\circ$.
\begin{itemize}
\item[(i)] Suppose that $\min I_0 = \min I$, $\max I_0 < \max I$, and $\rho^I_{\max I_0} < 1$.
Then $v_i^I < v_i^{I_0}$ for all $i \in I^\circ_0$.
\item[(ii)] Suppose that $\min I_0 > \min I$, $\max I_0 = \max I$, and $\rho^I_{\min I_0 - 1} < 1$.
Then $v_i^I > v_i^{I_0}$ for all $i \in I^\circ_0$.
\end{itemize}
Moreover, if either of the conditions in~(i) or~(ii) hold, then $\rho_i^I < \rho_i^{I_0}$ for all $i \in I_0^\circ$.
\end{lemma}
\begin{proof}
Suppose that the conditions in~(i) hold; for convenience, let $K = \max I_0$.
Since $\nu^{I_0}$ solves~\eqref{eq:reduced-balance} for $I_0$, and $\rho^{I_0}_i \leq 1$ for all $i \in I_0^\circ$,
it follows from Lemma~\ref{lem:reduced-system} that $\rho^{I_0}_i = \hrho_{I_0} (i)$, $i \in I_0^\circ$,
for $\hrho_{I_0}$ as defined at~\eqref{eq:rho-I-def}. 
Since $\nu^I$ solves~\eqref{eq:reduced-balance} for $I$, we have from~\eqref{eq:reduced-traffic-rhos} that
 the $\rho^I_i$, $i \in I_0^\circ$, solve the system~\eqref{eq:rho-I-I0},
with $\rho^I_{\min I-1} = \rho^I_{\min I_0 -1} = 1$ and $\rho^I_{K} < 1$. If we set $b_{K}' := \rho^I_K b_K$, then this system coincides with the reduced
traffic equation~\eqref{eq:reduced-traffic-rhos} over $I=I_0$, but with $b_K$ replaced by $b_K' < b_K$ (since $\rho^I_K < 1$).
Thus, if we define $\hrho'_{I_0}$ by~\eqref{eq:rho-I-def}, but with $b'_K$ in place of $b_K$,
then Lemma~\ref{lem:reduced-system} shows that $\rho^I_i = \hrho'_{I_0} (i)$ for all $i \in I_0^\circ$.

Suppose $I_0 = [\ell; m]$. In the formula~\eqref{eq:rho-I-def}
for $\hrho_{I_0} (j)$, $j \in I_0^\circ$, 
none of $\alpha (\ell; 1)$, $\beta(\ell; 1)$, \ldots, $\alpha (\ell; m-1)$, $\beta (\ell; m-1)$
contains $b_K = b_{\ell+m-1}$, which appears only in $\hv (\ell ; m)$. Moreover, as a function of $b_{\ell+m-1}$, we have $\hv (\ell; m) = A b_{\ell+m-1} -B$ for $A >0$,
so that $\hrho_{I_0} (j)$ is strictly increasing in $b_K$.
Since $b'_K < b_K$, it follows that $\rho^I_j = \hrho'_{I_0} (j) < \hrho_{I_0} (j) = \rho^{I_0}_j$ for all $j \in I_0^\circ$, as claimed.
In particular, $\rho^I_j < 1$ for all $j \in I_0^\circ$.

Let $L = \min I = \min I_0$. Then, from~\eqref{eq:v-I-def} and the fact that $\rho^I_L < \rho^{I_0}_L \leq 1$,
$v^{I_0}_{L} - v^{I}_{L} =   \bigl( \rho^{I_0}_{L} - \rho^{I}_L \bigr)  b_L > 0$,
and, by the appropriate version of~\eqref{eq:speed-neighbours},   $v^{I_0}_j = v^{I_0}_L > v^{I}_L = v^I_j$
for all $j \in I_0^\circ$. This completes the proof of the lemma in the case where the hypotheses of~(i) hold. 
The argument in the case that the hypotheses in~(ii) hold is similar.
\end{proof}

Now we can complete the proof of the main stability result, Theorem~\ref{thm:main}.

\begin{proof}[Proof of Theorem~\ref{thm:main}.]
Let $N \in \N$ and suppose that~\eqref{ass:positive-rates} holds.
Proposition~\ref{prop:general-Jackson} shows that	there exists a unique solution $\nu = (\nu_i)_{i \in [N]}$
	to the general traffic equation~\eqref{eq:general-balance}. Define $\rho_i := \nu_i / \mu_i$
	for every $i \in [N]$. Then the $\rho_i$ define uniquely an $n \in [N+1]$ and a partition $\Theta = (\theta_1, \ldots, \theta_n)$
	as follows. If $\rho_1 \geq 1$, set $\theta_1 = \{ 1\}$.
	Otherwise, set $\theta_1 = \{ 1, \ldots, k, k+1\}$ for the maximal $k \leq N$
	for which $\rho_k <1$. If $k = N$, this completes $\Theta$. Otherwise $k < N$ and $\rho_{k+1} \geq 1$.
	If $k+2 > N$ or $\rho_{k+2} \geq 1$, set $\theta_2 = \{ k+2\}$, else set $\theta_2 = \{ k+2, \ldots, k+1+m, k+2+m\}$
	for the maximal $m$ with $k + 1 + m \leq N$ and $\rho_{k+1+m} < 1$.
	Iterating this definition gives $\Theta = (\theta_1, \ldots, \theta_n)$ where, by
	construction, $\rho_i \in (0,1)$ for all $i \in \theta^\circ$, $\theta \in \Theta^\star$,
	and $\rho_{\max \theta} \geq 1$ provided $\max \theta \leq N$.
Lemma~\ref{lem:reduced-system-consistency} with $I = [N+1]$ and $I_0 = \theta \in \Theta^\star$
then shows that $\rho_i = \rho^{[N+1]}_i = \rho^\theta_i$ for all $i \in \theta^\circ$,
while Lemma~\ref{lem:reduced-system} says that $\rho^\theta = \hrho_\theta (i)$. Hence $\rho_i = \hrho_\theta (i)$ given by~\eqref{eq:rho-I-def} for $i \in \theta^\circ$. 
This proves~\eqref{eq:rho-rho-hat}.
	
For statement~\ref{thm:main-i},
observe that $S := \{ i \in [N] : \rho_i < 1\}$ as defined in Proposition~\ref{prop:general-Jackson}
is given by $S = \cup_{\theta \in \Theta^\star} \theta^\circ$, by the property established above that $\rho_i < 1$ if and only if $i \in \theta^\circ$
for some $\theta \in \Theta^\star$. Thus the first display in statement~\ref{thm:main-i} follows from Proposition~\ref{prop:general-Jackson},
using the fact that $\varpi_\theta (i)$ as defined in~\eqref{eq:limit-distribution-component}
is given by $\varpi_\theta (z_1, \ldots, z_k ) = \prod_{i \in \theta^\circ} \rho_i^{z_i} (1-\rho_i)$, where $\vert\theta\vert = k+1$ and $\rho_i = \hrho_\theta (i)$, by~\eqref{eq:rho-rho-hat}.
We have shown that $\lim_{t \to \infty} \IP ( \eta_j (t) \geq m ) = \rho_j^m$ for $j \in S$, and, by~\eqref{eq:eta-bound}, $\eta_j (t)$, $t \in \RP$, is uniformly integrable,
so that $\lim_{t \to \infty} \IE \eta_j (t) = \sum_{m \in \N} \rho_j^m = \rho_j / (1-\rho_j)$. With $R_\theta$ as defined at~\eqref{eq:R-def}, we have that $\IE R_\theta (t) = \IE \sum_{j \in \theta^\circ} (1+ \eta_j (t))$, by~\eqref{eq:eta-def},
and taking $t \to \infty$ yields~\eqref{eq:cloud-size}. This completes the proof of~\ref{thm:main-i}.

Consider statement~\ref{thm:main-ii}. For $\theta_\ell, \theta_r \in \Theta$ with $\ell < r$, we have from~\eqref{eq:eta-def} that 
$\min_{i \in \theta_r} X_i (t) - \max_{i \in \theta_\ell} X_i (t) \geq \eta_{\max \theta_\ell} (t)$. Moreover, from the statement after~\eqref{eq:rho-rho-hat},
we have $\rho_{\max \theta_\ell} \geq 1$. In the notation of Proposition~\ref{prop:general-Jackson}, this means $\max \theta_\ell \in U$, and hence,
by~\eqref{eq:mass-escape}, $\lim_{t \to \infty} \IP ( \eta_{\max \theta_\ell} (t) \leq B ) = 0$ for all $B < \infty$.
This proves~\eqref{eq:cloud-separation}.
	
	For statement~\ref{thm:main-iii}, existence of the $v_i$ follows from Theorem~\ref{thm:speeds}. Lemma~\ref{lem:speeds-algebra} shows that, for each $\theta \in \Theta$, $v_i$ is the same for all $i \in \theta^\circ$;
	call the common value $\hv (\theta)$. The fact that $\hv (\theta)$ satisfies~\eqref{eq:hv-def} and~\eqref{eq:hv-l-m} follows from Lemma~\ref{lem:reduced-system}. This completes the proof of the theorem.
\end{proof}

The following theorem 
will be our main tool to prove Theorem~\ref{thm:algorithm}:
it takes the results of Lemmas~\ref{lem:reduced-system}--\ref{lem:stable-extension}
and presents them in terms of pairwise comparisons of adjacent candidate stable clouds,
tailored to the structure of Algorithm~\ref{alg:partition}.
In particular, Theorem~\ref{thm:block-merge} shows that if one has two adjacent candidate stable clouds, in which the cloud to the left has a
greater intrinsic speed, then the union of the two candidate stable clouds is itself a candidate stable cloud.

\begin{theorem}
\label{thm:block-merge}
	Let $N \in \N$ and suppose that~\eqref{ass:positive-rates} holds.
Let $I_1$ and $I_2$ be disjoint discrete interval subsets of $[N+1]$,
with $K := \max I_1 = \min I_2 -1$. 
Let $\nu^{I_k}$ be the solution to the reduced traffic equation~\eqref{eq:reduced-balance} with $I = I_k$, $k \in \{1,2\}$.
Define $\rho^{I_k}_i := \nu_i^{I_k}/\mu_i$.

Suppose that $\rho^{I_1}_i \leq 1$ for all $i \in I_1^\circ$ and
$\rho^{I_2}_i \leq 1$ for all $i \in I_2^\circ$.
Recall the definition of $\hv$ from~\eqref{eq:hv-def}. Let $I := I_1 \cup I_2$, which is also
a discrete interval subset of $[N+1]$.
Then the following hold.
\begin{thmenumi}[label=(\roman*)]
\item If $\hv({I_1}) > \hv(I_2)$, then $\rho^I_i < 1$ for all $i \in I^\circ$,
and $\hv ({I_1}) > \hv ({I}) > \hv ({I_2})$. Moreover,
\begin{equation}
\label{eq:rhos-decrease}
\rho^I (i) < \begin{cases} \rho^{I_1}_i & \text{if } i \in I_1^\circ , \\
1 & \text{if } i = K , \\
\rho^{I_2}_i & \text{if } i \in I_2^\circ .\end{cases} \end{equation}
\item  If $\hv ({I_1}) \leq \hv ({I_2})$, then  
\begin{equation}
\label{eq:rhos-equal}
 \rho^I_i = \begin{cases} \rho^{I_1}_i & \text{if } i \in I_1^\circ , \\
\rho^{I_2}_i  & \text{if }   i \in I_2^\circ .\end{cases} \end{equation}
Moreover, if $\hv ({I_1}) < \hv ({I_2} )$, then $\rho^{I}_K > 1$,
while if $\hv ({I_1}) = \hv ({I_2})$, then $\rho^I_K = 1$.
\end{thmenumi}
\end{theorem}
\begin{proof} 
Let $I_1$ and $I_2$ be disjoint discrete interval subsets of $[N+1]$,
with $K := \max I_1 = \min I_2 -1$. 
Let $\nu^{I_k}$ be the solution to the reduced traffic equation~\eqref{eq:reduced-balance} with $I = I_k$, $k \in \{1,2\}$.
Define $\rho^{I_k}_i := \nu_i^{I_k}/\mu_i$, and define $v^{I_k}$ by~\eqref{eq:v-I-def}. 
Suppose that $\rho^{I_1} (i) \leq 1$ for all $i \in I_1^\circ$ and
$\rho^{I_2}_i \leq 1$ for all $i \in I_2^\circ$. 
Then 
Lemma~\ref{lem:reduced-system}
shows that 
\begin{equation}
\label{eq:two-speeds} v^{I_1}_i = \hv(I_1) \text{ for all } i \in I_1, 
\text{ and } v^{I_2}_i = \hv (I_2) \text{ for all } i \in I_2. \end{equation} 

Consider also the solution $\nu^I$ to the reduced traffic equation~\eqref{eq:reduced-balance} over~$I := I_1 \cup I_2$. 
Suppose first that $\rho^I ( K ) \geq 1$, i.e., $\nu^I_{K} \geq \mu_{K}$. 
Then Lemma~\ref{lem:reduced-system-consistency},
applied first with $I_0 = I_1$ and second with $I_0=I_2$,
 implies that~\eqref{eq:rhos-equal} holds. From~\eqref{eq:v-I-def}, this means that
$v^I_i = v^{I_1}_i$  for all $i \in I_1$ and $v^I_i =  v^{I_2}_i$  for all $i \in I_2$.
Theorem~\ref{thm:speeds} and~\eqref{eq:speed-neighbours} (appropriately adapted) imply that if $\rho^I ( K ) = 1$, we have $v^I_K = v^I_{K+1}$, while if 
$\rho^I_K >1$, we have $v^I_K < v^I_{K+1}$. Thus it follows from~\eqref{eq:two-speeds} that
\begin{align}
\label{eq:implication-1} 
\rho^I_K & > 1 \text{ implies } \hv ({I_1}) < \hv ({I_2}) ; \\
\label{eq:implication-2} 
\rho^I_K  & = 1 \text{ implies  } \hv ({I_1}) = \hv ({I_2}) .
\end{align}
On the other hand, 
Lemma~\ref{lem:stable-extension} shows that 
\begin{align}
\label{eq:implication-3} 
\rho^I_K & < 1 \text{ implies } \hv ({I_1}) > \hv (I) > \hv ({I_2}) ,
\end{align}
and the final statement in Lemma~\ref{lem:stable-extension} yields~\eqref{eq:rhos-decrease}. 
Combining~\eqref{eq:implication-1}--\eqref{eq:implication-3}, we obtain
\begin{align}
\label{eq:equivalence-1} 
\rho^I_K & > 1 \text{ if and only if } \hv ({I_1}) < \hv ({I_2}) ; \\
\label{eq:equivalence-2} 
\rho^I_K & = 1 \text{ if and only if  } \hv ({I_1}) = \hv ({I_2}) ; \\
\label{eq:equivalence-3} 
\rho^I_K & < 1 \text{ if and only if } \hv ({I_1}) > \hv (I) > \hv ({I_2}) .
\end{align}
The equivalences~\eqref{eq:equivalence-1}--\eqref{eq:equivalence-3} complete the proof of the theorem. 
\end{proof}

We now complete the proofs of Theorem~\ref{thm:algorithm} and the corollaries presented in Section~\ref{sec:results}.

\begin{proof}[Proof of Theorem~\ref{thm:algorithm}.]
We claim that 
Theorem~\ref{thm:block-merge}(i) and an induction shows that, 
at each step of Algorithm~\ref{alg:partition},
the ordered partition $\Theta^\kappa$ is such that
\begin{align}
\label{eq:stability-induction}
 \rho_i^\theta < 1 \text{ for all } i \in \theta^\circ \text{ and all } \theta \in \Theta^\kappa \text{ with } \vert \theta \vert \geq 2.
\end{align}
The base case of~\eqref{eq:stability-induction}, $\kappa = 0$, is trivial since the initial partition consists only of singletons.
For the inductive step, the inductive hypothesis deals trivially with
any $\theta$ which is unchanged from $\Theta^\kappa$ to $\Theta^{\kappa+1}$. Consider adjacent $\theta^\kappa_{j}, \theta^\kappa_{j+1} \in \Theta^\kappa$
with $\hv ( \theta^\kappa_j ) > \hv (\theta^\kappa_{j+1} )$, and define the merger $\theta^{\kappa+1}_j = \theta^\kappa_j \cup \theta^\kappa_{j+1}$.
Then Theorem~\ref{thm:block-merge}(i) and the inductive hypothesis completes the proof of~\eqref{eq:stability-induction}. 

From~\eqref{eq:stability-induction}, the output of Algorithm~\ref{alg:partition}  
is an ordered partition $\Theta$ such that
$\rho_i^\theta < 1$ for all $i \in \theta^\circ$ and all $\theta \in \Theta^\star$.
Moreover, 
if $\Theta = (\theta_1, \ldots, \theta_n)$, the stopping condition for Algorithm~\ref{alg:partition}
implies that $\hv ( \theta_1 ) \leq \cdots \leq \hv (\theta_n)$.
Thus Theorem~\ref{thm:block-merge}(ii) shows that for every adjacent $\theta_j, \theta_{j+1} \in \Theta$,
one has $\rho^\theta_K \geq 1$ where
$K = \max \theta_j$ and $\theta = \theta_j \cup \theta_{j+1}$.
Then Lemma~\ref{lem:reduced-system-consistency}
completes the verification of~\eqref{eq:rho-rho-hat}, showing that the $\Theta$ produced by Algorithm~\ref{alg:partition} is the cloud partition
characterized in Theorem~\ref{thm:main}. 
\end{proof}

\begin{proof}[Proof of Corollary~\ref{cor:singletons}.]
Theorem~\ref{thm:main} shows that $\Theta$ consists solely of singletons if and only if $\rho_i \geq 1$ for all $i \in [N]$, where $\rho_i =\nu_i/ \mu_i$ and
 $\nu$ is the unique solution to~\eqref{eq:general-balance}. In particular, if
$\Theta$ consists only of singletons, then Theorem~\ref{thm:main}\ref{thm:main-iii} with~\eqref{eq:hv-def} implies that $v_i = \hv ( \{ i \} ) = b_i - a_i$ must satisfy $v_{i+1} \geq v_i$ for all $i$,
which yields~\eqref{eq:singleton-condition}.
On the other hand, if~\eqref{eq:singleton-condition} holds, then Algorithm~\ref{alg:partition} will terminate immediately, and Theorem~\ref{thm:algorithm} implies that  $\Theta$ consists solely of singletons.
The final sentence in Corollary~\ref{cor:singletons} follows from Theorem~\ref{thm:main}\ref{thm:main-ii} and~\ref{thm:main-iii}.
\end{proof}

\begin{proof}[Proof of Corollary~\ref{cor:stable}.]
Theorem~\ref{thm:main} combined with Proposition~\ref{prop:stable-Jackson}
shows that $\Theta = ([N+1])$ consists of a single part if and only if $\rho_i < 1$ for all $i \in [N]$, where $\rho_i = \nu_i/\mu_i$ and
 $\nu$ is the unique solution to~\eqref{eq:stable-balance}, i.e., $\rho_i$ is given by~\eqref{eq:rho-stable}.
If $\hv_{N+1} = \hv ([N+1])$ as given at~\eqref{eq:hv-def}, then~\eqref{eq:v-stable} holds. The limit statement~\eqref{eq:stable-limit}
follows from Theorem~\ref{thm:main}\ref{thm:main-i}, while Theorem~\ref{thm:main}\ref{thm:main-iii} gives the $\lim_{t \to \infty} t^{-1} X_i (t) = \hv_{N+1}$ for all $i \in [N+1]$.
\end{proof}

\begin{proof}[Proof of Corollary~\ref{cor:one-direction}.]
Suppose that $\Theta = (\theta_1, \ldots, \theta_n)$ is the cloud partition
described in Theorem~\ref{thm:main}. 
Then Theorem~\ref{thm:main}\ref{thm:main-iii}
says that $\min_{i \in [N+1]} v_i = v_1 = \hv (\theta_1)$, where $\hv$ is defined at~\eqref{eq:hv-def}.
Suppose that $\theta_1 = [1;m]$ for $m \in [N+1]$. Then $\hv (\theta_1 ) = \hv (1;m) > 0$ if and only if $\alpha (1;m) < 1$,
where $\alpha$ is defined by~\eqref{eq:alpha-beta-def}. Thus, $\hv (\theta_1 ) > 0$ if and only if $\prod_{i=1}^m a_i < \prod_{i=1}^m b_i$.
Thus condition~(ii) in the corollary implies~(i). 

Conversely, suppose that $\hv (\theta_1 ) >0$.
If $\theta_1 = [1;m]$ for $m \geq 2$, then $\hv (\theta_1^\circ ) > \hv (\theta_1) > 0$,
by applying Theorem~\ref{thm:block-merge} with $I_1 = \theta_1^\circ = \{1, \ldots, m-1\}$ and $I_2 = \{m\}$,
using~\eqref{eq:rho-rho-hat} to characterize the stability of~$\theta_1$. Iterating Theorem~\ref{thm:block-merge},
it follows that $\hv (1;k ) > 0$ for all $k \in [m]$. 
Suppose that~(ii) fails; then there is some $r \geq m$
such that  $\prod_{i=1}^k a_i < \prod_{i=1}^k b_i$ for all $k \leq r$, but
 $\prod_{i=1}^{r+1} a_i \geq \prod_{i=1}^{r+1} b_i$. In particular, $a_{r+1} > b_{r+1}$,
$a_{r+1} a_r > b_{r+1} b_r$, \ldots, $a_{r+1} \cdots a_2 > b_{r+1} \cdots b_2$.
By~\eqref{eq:hv-def}, this means that $\hv ( I ) < 0$ for any $I$ with $\max I = r+1$ and $\min I \geq 2$.
Considering Algorithm~\ref{alg:partition}, if $\theta \in \Theta$ is such that $r +1 \in \theta$,
it follows that $\hv (\theta ) < 0$ too. 
But the speeds of stable clouds increase from left to right, so this is not possible, giving a contradiction.
Hence~(i) implies~(ii).
\end{proof}

\section{Central limit theorem for the stable system}
\label{sec:clt}

In this section we prove Theorems~\ref{thm:clt} and~\ref{thm:two-particle-clt}, the
central limit theorems in the case of a stable system.

\begin{proof}[Proof of Theorem~\ref{thm:clt}.]
Suppose that $\rho_i < 1$ for all $i \in [N]$, where $\rho_i$ is given by~\eqref{eq:rho-stable},
and define $\hv_{N+1}$ by~\eqref{eq:v-stable}. 
Consider the processes $\xi$ (on $\Z \times \ZP^N$) and $\eta$ (on $\ZP^N$)
defined through~\eqref{eq:eta-def} and~\eqref{eq:xi-def};
by Corollary~\ref{cor:stable} and~\eqref{eq:limit-distribution-component}, 
\[
\lim_{t \to \infty} \IP ( \eta (t) = z ) = \prod_{i=1}^N \rho_i^{z_i} (1-\rho_i) = \varpi_{[N+1]} ( z ) ,  \text{ for all } z=  (z_1, \ldots, z_N ) \in \ZP^N.
\]
For ease of notation, write $\tX_1 (t) := X_1 (t) - \hv_{N+1} t$.
As described in Remarks~\ref{rems:clt}\ref{rems:clt-b},
from the results of~\cite{ak} one can directly deduce the marginal central limit theorem
\begin{equation}
\label{eq:X1-clt} 
(\sigma^2 t)^{-1/2} \tX_1 (t) \tod \cN (0, 1 ) , \end{equation}
for a constant $\sigma^2 \in (0,\infty)$.

Take $s, t$ with $0 < t-s < t < \infty$,
and let $E_{s,t} (w) := \{ (\sigma^2 t)^{-1/2} \tX_1 (t-s) \leq w \}$, for $w \in \R$.
Fix $\eps >0$.  Choose a finite set $A \subset \ZP^N$ such that $\varpi_{[N+1]} (A) > 1-\eps$.
Since $\xi$ and $\eta$ defined by~\eqref{eq:eta-def} and~\eqref{eq:xi-def} are both Markov,
we have that
$\IP ( \eta (t) = z \mid \xi (t-s) )$ depends on $\xi(t-s)$ only through $\eta (t-s)$, and not $\tX_1 (t-s)$. Hence
\begin{align*}
& {}  \IP [ E_{s,t} (w) \cap \{ \eta(t) = z \} ] \\
& \quad {} \leq \IP [ \eta (t-s) \notin A ] + \IE \bigl[ \IP [ \eta(t) = z  \mid \xi (t-s) ] \1 { \eta(t-s) \in A} \2 { E_{s,t} (w)} \bigr] \\
& \quad {}  = \IP [ \eta (t-s) \notin A ] + \IE \bigl[ \IP [ \eta(t) = z  \mid \eta (t-s) ] \1 { \eta(t-s) \in A} \2 { E_{s,t} (w)} \bigr] .\end{align*}
 Then
we can choose $s_0$ large enough (depending on $\eps$ and $z$) such that, for all $s,t$ with $s \geq s_0$ and $t-s \geq s_0$,
$ \IP [ \eta (t-s) \in A ] \geq 1 -2 \eps$, and 
\[  \bigl\vert \IP [ \eta(t) = z  \mid \eta (t-s) ] - \varpi_{[N+1]} (z) \bigr\vert \1 { \eta(t-s) \in A } \leq \eps .\]
The preceding argument, together with a similar argument for a lower bound, shows that for every $\eps >0$, $z \in \ZP^N$, and $w \in \R$,
we may choose and fix $s$ (large) such that, for all $t$ with $t-s$ large enough,
\begin{align*}
\bigl\vert \IP [ E_{s,t} (w) \cap \{ \eta(t) = z \} ] -\varpi_{[N+1]} (z)  \IP [ E_{s,t} (w) ] \bigr\vert \leq \eps .\end{align*}
Now from~\eqref{eq:X1-clt} we have that $(\sigma^2(t-s))^{-1/2} \tX_1 (t-s)$ converges in distribution to $\cN (0, \sigma^2 )$ as $t \to \infty$.
For fixed $s$, $\lim_{t \to \infty} \frac{t}{t-s} = 1$, and so $\lim_{t \to \infty} \IP [ E_{s,t} (w) ] = \Phi (w)$, the standard normal distribution function.
Moreover, it is not hard to see that, since $X_1$ has uniformly bounded jump rates with increments $\pm 1$,
\[ \lim_{t \to \infty} t^{-1/2} \IE \big\vert \tX (t) - \tX (t-s) \bigr\vert = 0 ,\] 
for fixed $s$.
Hence, for fixed $z \in \ZP^N$, for all $t$ large enough it holds that
\[ \bigl\vert \IP [ E_{s,t} (w) \cap \{ \eta(t) = z \} ] - \IP [ \{  (\sigma^2 t)^{-1/2} \tX_1 (t) \leq w \} \cap \{ \eta(t) = z \} ] \bigr\vert \leq \eps , \]
using continuity of $\Phi$.
Combining these bounds we conclude that, for every $\eps >0$ and all $t$ sufficiently large
\[ \bigl\vert  \IP [ \{  (\sigma^2 t)^{-1/2} \tX_1 (t) \leq w \} \cap \{ \eta(t) = z \} ] - \varpi_{[N+1]} (z) \Phi (w) \bigr\vert \leq \eps .\]
This completes the proof.
\end{proof}

Finally, we turn to the proof of Theorem~\ref{thm:two-particle-clt}, the main part of which is
to compute the limiting variance $\sigma^2$ in the $N=1$ case of Theorem~\ref{thm:clt}.

\begin{proof}[Proof of Theorem~\ref{thm:two-particle-clt}.]
Take $N=1$, so the system consists of $N+1 = 2$ particles, with locations $X_1(t) < X_2(t)$, and a single
gap process, given for $t \in \RP$ by $\eta(t) := \eta_1 (t) = X_2(t) - X_1(t) - 1 \in \ZP$.
In this case, $\eta$ is a classical M/M/1 queue,
with arrival rate $\lambda := a_1 + b_2$ and service rate $\mu := a_2 + b_1$. 
 Under the conditions of Theorem~\ref{thm:two-particle-clt}, we have $\lambda < \mu$, i.e., the M/M/1 queue is stable.
Then $\eta$ has a unique geometric stationary distribution; denote by $\tIP$ and $\tIE$
probability and expectation for an initial distribution such that $\eta(0)$ is in stationarity, and, for definiteness $X_1 (0) = 0$.

Theorem~\ref{thm:clt}
implies that~\eqref{eq:two-particle-clt} holds, where~\eqref{eq:v-stable} yields $\hv_{N+1} = \hv_2$ as in~\eqref{eq:sigma-two-particle};
it remains to verify that $\sigma^2$ is given by~\eqref{eq:sigma-two-particle}. Since  Theorem~\ref{thm:clt} applies
for any fixed initial state $X(0) \in \bbX_{N+1}$, it also holds for any initial distribution; hence also  under $\tIP$,
when $\eta(0)$ is stationary. Thus to complete the proof of Theorem~\ref{thm:two-particle-clt}, it suffices to prove
\begin{equation}
\label{eq:stationary-variance}
 \lim_{t \to \infty} t^{-1} \tVar ( X_1 (t) )   = \frac{a_1 a_2 + b_1 b_2}{a_2+b_1}.
\end{equation}

For $t \in \RP$, let 
\[ A(t) := \# \{ s \in [0,t] : \eta (s) = \eta(s-) + 1 \}, \text{ and } D(t) := \# \{ s \in [0,t] : \eta (s) = \eta(s-) -1\} ,\]
denote, respectively, the number of arrivals and departures from the queue~$\eta$ up to time $t$.
Then $A(t) = A_1 (t) + A_2(t)$, where
\begin{align*}
 A_1 (t) & := \#\{ s \in [0,t] : X_1 (s) = X_1(s-) - 1 \},\\
 A_2 (t) & := \#\{ s \in [0,t] : X_2 (s) = X_2(s-) + 1 \} ,
\end{align*}
count arrivals representing particle $1$ jumping left and particle $2$ jumping right, respectively. Similarly,
$D(t) = D_1 (t) + D_2 (t)$, where
\begin{align*}
 D_1 (t) & := \#\{ s \in [0,t] : X_1 (s) = X_1(s-) + 1 \}, \\
D_2 (t) & := \#\{ s \in [0,t] : X_2 (s) = X_2(s-) - 1 \} .
\end{align*}
We have the representations $\eta(t) - \eta(0) = A(t) - D(t)$, 
\[ X_1 (t) - X_1(0) = D_1(t) - A_1(t), \text{ and } X_2 (t) - X_2(0) = A_2(t) - D_2 (t) , \text{ for } t \in \RP, \]
where $X_2(0) = X_1(0) + 1 + \eta(0)$. 
Conditional on $\eta[0,t] := (\eta_s)_{s \in [0,t]}$, we have that $A_1(t) \sim \text{Bin} ( A(t), p_1)$, $A_2(t) \sim \text{Bin} (A(t), 1-p_1)$,
where $p_1 := a_1/(a_1+b_2)$, and  $D_1(t) \sim \text{Bin} ( D(t), p_2)$, $D_2(t) \sim \text{Bin} (D(t), 1-p_2)$,
where $p_2 := b_1/(a_2+b_1)$. Moreover, given $\eta[0,t]$,
$(A_1(t),A_2(t))$ is independent of $(D_1(t),D_2(t))$. 
Hence we see that
\begin{align*}
\IE ( X_1 (t) - X_1 (0) \mid \eta[0,t] ) & = - p_1 A(t) + p_2 D(t), \\
\IE ( X_2 (t) - X_2 (0) \mid \eta[0,t] ) & = (1- p_1) A(t) - (1- p_2) D(t), \\
\Var ( X_i (t) \mid \eta[0,t] ) & = p_1 (1-p_1) A(t) + p_2 (1-p_2) D(t) , \text{ for } i \in \{1,2\}.\end{align*}
Now, $A(t) \sim \text{Po} (\lambda t)$, and, by Burke's theorem~\cite{burke}, $D(t) \sim \text{Po} (\lambda t)$ under $\tIP$.
It follows that 
\begin{align*} \tIE \left[ \Var ( X_i (t) \mid \eta[0,t] ) \right] & = \left[ p_1 (1-p_1) + p_2 (1-p_2) \right] \lambda t, \text{ and} \\
\tVar \left[ \IE ( X_1 (t) \mid \eta[0,t] ) \right] & = \left[ p_1^2 + p_2^2 \right] \lambda t - 2p_1 p_2 \tCov ( A(t), D(t) )  .\end{align*}
Note that
\[ A (t) D (t) = \frac{1}{2} \left[ A (t)^2 + D(t)^2 - (A (t) - D (t ))^2 \right] ,\]
so
\[ \tIE [ A(t) D(t) ] = \lambda^2 t^2 + \lambda t - \frac{1}{2} \tIE [ (\eta(t)-\eta(0))^2 ]. \]
That is,
\[ \tCov ( A(t), D(t) ) = \lambda t - \frac{1}{2} \tVar ( \eta(t) - \eta(0) ) .\]
Under the stationary distribution, $\eta$ has finite variance, so $\sup_{t \in \RP} \tVar ( \eta(t) - \eta(0) )  < \infty$. Hence,
using the total variance formula and collecting the above computations, we get
\begin{align*} \tVar ( X_1 (t) ) & = \tIE \left[ \Var ( X_1 (t) \mid \eta[0,t] ) \right] + \tVar \left[ \IE ( X_1 (t) \mid \eta[0,t] ) \right]  \\
& = \left[ p_1 (1-p_1) + p_2 (1-p_2) + (p_1 - p_2)^2 \right] \lambda t + O(1) \\
& = \left[ p_1 + p_2 - 2 p_1 p_2 \right] \lambda t + O(1) \\
& = \frac{a_1 a_2 + b_1 b_2}{a_2+b_1} t + O(1) .\end{align*}
From here, we verify~\eqref{eq:stationary-variance}, to conclude the proof.
\end{proof}

\section{Discussion and connection to continuum models}
\label{sec:discussion}

\subsection{Diffusions with rank-based interactions}
\label{sec:diffusions}

Continuum (diffusion) models, in which $N+1$ particles perform
Brownian motions with drifts and diffusion coefficients determined by their ranks, 
have been extensively studied in recent years, and include the \emph{Atlas model} and its relatives:
see~\cite{pp,bfk,cdss,ik,ipbkf,kpsAIHP,tsai,sarantsevAIHP,sarantsevEJP} and references therein. 
An informative recent overview of the literature is given in~\cite[\S 2]{sarantsevAIHP}.
In this section we describe these models using some stochastic differential equations (SDEs) rather informally, so that we can make comparisons;
precise formulations, existence and uniqueness results can be found in the references cited.

In the early
versions of these models (e.g.~\cite{bfk,pp}), the particles move independently, and are allowed to pass through each other, but exchange drift and diffusion parameters
when they do so. For $t\in \RP$, denote by $y_i (t)$ the location of particle with label $i \in [N+1]$,
and let $x_i (t)$ denote the location of the particle with \emph{rank} $i$ ($i=1$ being the leftmost particle),
so $x_1 (t) \leq \cdots \leq x_{N+1} (t)$ is an ordering of $y_1 (t), \ldots, y_{N+1} (t)$. 
Also let $r_i (t)$ denote the rank of particle with label $i$. 
Let $u_1, \ldots, u_{N+1} \in \R$ and $\sigma_1, \ldots, \sigma_{N+1} \in (0,\infty)$ be collections
of drift and volatility parameters. Let $W_1, \ldots, W_{N+1}$ be independent, standard $\R$-valued Brownian motions. 
 Suppose that
\begin{equation}
\label{eq:SDE-labelled-symmetric}
\ud y_i (t) = u_{r_i(t)} \ud t + \sigma_{r_i(t)} \ud W_i (t) , \text{ for } i \in [N+1], \end{equation}
i.e., particle $i$ performs Brownian motion with drift and volatility parameters determined by its rank;
if $u_j \equiv u$ and $\sigma_j \equiv \sigma$ are constant, the particles are independent, otherwise, there is interaction
mediated by the ranks.
The Atlas model~\cite{bfk,ipbkf} has $u_1 > 0$, $u_2 = \cdots = u_{N+1} = 0$, and $\sigma_1 = \cdots = \sigma_{N+1} =1$; cf.~Example~\ref{ex_N_sheep} above.

 In terms of the ordered particles, it can be shown (e.g.~\cite[Lem.~1]{ipbkf}) that,
for independent standard Brownian motions $\tW_1, \ldots, \tW_{N+1}$,  
\begin{equation}
\label{eq:SDE-ordered-symmetric}
\ud x_i (t) = u_{i} \ud t + \sigma_{i} \ud \tW_i (t) + \frac{1}{2} \ud L^{i} (t) - \frac{1}{2} \ud L^{{i+1}} (t) , \text{ for } i \in [N+1],
  \end{equation}
where $L^{i}(t)$ is the local time at~$0$ of $x_i-x_{i-1}$ up to time~$t$,
with the convention $L^1 \equiv L^{N+2} \equiv 0$. The local-time terms in~\eqref{eq:SDE-ordered-symmetric}
maintain the order of the~$x_i$, while each particle~$x_i$ has its own intrinsic drift and diffusion coefficients,
so the continuum process $x = (x_1, \ldots,x_{N+1})$ shares key features with our particle process $X = (X_1, \ldots, X_{N+1})$ described in Section~\ref{sec:intro}. However, it turns out that there are important differences inherent in the reflection mechanisms; 
before expanding on this,  we introduce an  extension of~\eqref{eq:SDE-ordered-symmetric}
to more general reflections.

In~\cite{kpsAIHP}, it is demonstrated that a natural extension to~\eqref{eq:SDE-labelled-symmetric} is
\begin{align}
\label{eq:SDE-labelled-asymmetric}
\ud y_i (t) & = u_{r_i(t)} \ud t + \sigma_{r_i(t)} \ud W_i (t) 
+ \left( q^+_{r_i(t)} - \tfrac{1}{2} \right) \ud L^{r_i(t)} (t) -  \left( q^-_{r_i(t)} - \tfrac{1}{2} \right) \ud L^{r_i(t)+1} (t)
, \end{align}
where parameters $q^-_j, q^+_j \in (0,1)$, $j \in [N+1]$,
 satisfy $q_{j+1}^+ + q_j^- = 1$ for all $j \in [N]$. Equivalently, in terms of the $x_i$, the dynamics~\eqref{eq:SDE-labelled-asymmetric}
 amounts to
\begin{equation}
\label{eq:SDE-ordered-asymmetric}
\ud x_i (t) = u_{i} \ud t + \sigma_{i} \ud \tW_i (t) + q_i^+ \ud L^{i} (t) - q_i^- \ud L^{{i+1}} (t) . \end{equation}
The \emph{symmetric} case of~\eqref{eq:SDE-labelled-asymmetric}--\eqref{eq:SDE-ordered-asymmetric} where $q_i^- \equiv q_i^+ \equiv 1/2$ for all $i$ reduces
 to~\eqref{eq:SDE-labelled-symmetric}--\eqref{eq:SDE-ordered-symmetric};
the general case is called \emph{asymmetric}~\cite{sarantsevAIHP}.

As far as we are aware, long-time stability of the continuum system described above has been studied only as far as classifying whether or not the whole system is stable, rather than identifying a full cloud decomposition as we do in Theorem~\ref{thm:main}. 
In the symmetric case, the criterion for stability was 
obtained in~\cite[Thm.~8]{pp} (for the case of constant diffusion coefficients) and~\cite{bfk,ipbkf},
and in the stable case the asymptotic speed was derived; we summarize these results as follows, referring to e.g.~\cite[Prop.~2.2]{sarantsevAIHP} and~\cite[Prop.~2]{ipbkf} for precise formulations.

\begin{theorem}
\label{thm:diffusion}
For the system with symmetric collisions defined by~\eqref{eq:SDE-ordered-symmetric},
write $U_k := k^{-1} \sum_{j=1}^k u_j$ for $k \in [N+1]$. The system is stable if and only if $U_k > U_{N+1}$ for all $k \in [N]$.
Moreover, if stable, 
there is a strong law of large numbers with limiting speed $U_{N+1}$. 
\end{theorem}

The stability condition here is comparable to~\eqref{eq:speeds-condition}, but different in detail; we elaborate on this in the next section. 
In the general (asymmetric) setting of~\eqref{eq:SDE-ordered-asymmetric}, the criterion for stability is also known:
see e.g.~\cite[Prop.~2.1]{sarantsevAIHP}, and involves the inverse of a tridiagonal matrix, reminiscent of the Jackson criterion (see Remarks~\ref{rems:stable}).
In the next section, we turn to the relationship between the continuum model and lattice models, which is multifaceted but has been partially elaborated in~\cite[\S 3]{kpsAIHP}.

\subsection{Particle models with elastic collisions}
\label{sec:extensions}

It turns out that a suitable discrete-space parallel to the continuum model in Section~\ref{sec:diffusions}
is obtained by modifying the collision mechanism for the model in Section~\ref{sec:intro} from the exclusion rule to an \emph{elastic} collision rule.
We sketch here only the main idea; a pertinent discussion can also be found in~\cite[\S 3]{kpsAIHP}.

In the case $N=1$, two particles have jump rates $a_1, b_1$ and $a_2, b_2$. Under the exclusion rule,
if the particles are adjacent, the total activity is reduced (particle $1$ jumps left at rate $a_1$, while particle $2$ jumps right at rate $b_2$).
Suppose instead the total activity is maintained, by particles transferring `momentum' to their neighbours when they attempt to jump, so that particle $1$ jumps left at rate $a_1+a_2$,
while particle $2$ jumps right at rate $b_1 + b_2$. We call this the \emph{elastic} collision rule; see Figure~\ref{fig:elastic} for an illustration for a 6-particle configuration.
A slight variation on this is to modify the process so that `collisions' mean that two particles occupy the same site, rather than adjacent sites. The elastic interaction is then equivalent
to the two particles being permitted to pass each other, exchanging jump parameters as they do so (similarly to the diffusion model described in Section~\ref{sec:diffusions}).

\begin{figure}[t]
\centering
\scalebox{0.85}{
 \begin{tikzpicture}[domain=0:1, scale=1.0]
\draw[dotted,<->] (0,0) -- (13,0);
\node at (13.4,0)       {$\Z$};
\draw[black,fill=white] (1,0) circle (.5ex);
\draw[black,fill=black] (2,0) circle (.5ex);
\node at (2,-0.6)       {\small $X_1 (t)$};
\draw[black,fill=white] (3,0) circle (.5ex);
\draw[black,fill=black] (4,0) circle (.5ex);
\node at (4,-0.6)       {\small $X_2 (t)$};
\draw[black,fill=black] (5,0) circle (.5ex);
\node at (5,-0.6)       {\small $X_3 (t)$};
\draw[black,fill=white] (6,0) circle (.5ex);
\draw[black,fill=white] (7,0) circle (.5ex);
\draw[black,fill=white] (8,0) circle (.5ex);
\draw[black,fill=black] (9,0) circle (.5ex);
\node at (9,-0.6)       {\small $X_4 (t)$};
\draw[black,fill=black] (10,0) circle (.5ex);
\node at (10,-0.6)       {\small $X_5 (t)$};
\draw[black,fill=black] (11,0) circle (.5ex);
\node at (11,-0.6)       {\small $X_6 (t)$};
\draw[black,fill=white] (12,0) circle (.5ex);
\node at (1.3,0.6)       {\small rate $a_1$};
\node at (2.5,0.6)       {\small $b_1$};
\draw[black,->,>=stealth] (2,0) arc (30:141:0.58);
\draw[black,->,>=stealth] (4,0) arc (30:141:0.58);
\node at (3.5,0.6)       {\small $a_2+a_3$};
\node at (5.5,0.6)       {\small $b_2+b_3$};
\draw[black,->,>=stealth] (9,0) arc (30:141:0.58);
\node at (8.5,0.6)       {\small $a_4+a_5+a_6$};
\node at (11.5,0.6)       {\small $b_4+b_5+b_6$};
\draw[black,->,>=stealth] (2,0) arc (150:39:0.58);
\draw[black,->,>=stealth] (5,0) arc (150:39:0.58);
\draw[black,->,>=stealth] (11,0) arc (150:39:0.58);
\end{tikzpicture}}
\caption{Schematic of a lattice model with $N+1 = 6$ particles and an elastic collision rule. Filled circles represent particles, empty circles represent unoccupied lattice sites. Within each block of adjacent particles, only the particles at the extreme left or right can jump, but their jump rates are the sum of the intrinsic rates for all particles in the block. An essentially equivalent model is obtained by permitting particles to pass through each other, exchanging parameters as they do so.}
\label{fig:elastic}
\end{figure}
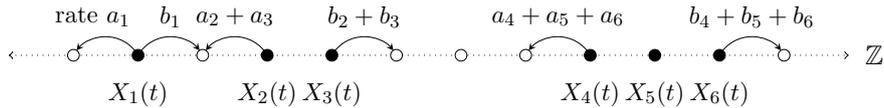

The elastic model is more homogeneous than the exclusion model, in the following sense: regardless of the present configuration, the rate at which
a single coordinate changes by $+1$ is constant (namely $B_{N+1} :=\sum_{i=1}^{N+1} b_i$), as is the rate at which a single coordinate changes by $-1$  ($A_{N+1}:=\sum_{i=1}^{N+1} a_i$).
Another way of saying this is that the centre of mass of the particle system performs a continuous-time simple random walk on $(N+1)^{-1} \Z$
with left jump rate $A_{N+1}$ and right jump rate $B_{N+1}$; thus it has speed $U_{N+1} :=(B_{N+1} - A_{N+1})/(N+1)$. Furthermore, considering just the leftmost $k \in [N+1]$ particles,
their centre of mass has jump rates $A_k$, $B_k$, and speed $U_k = (B_k-A_k)/k$. By a similar reasoning to Algorithm~\ref{alg:partition}, it becomes very plausible that the condition for stability of
the system should be $U_k > U_{N+1}$ for all $k \in [N]$, exactly as in Theorem~\ref{thm:diffusion}. It is natural to conjecture:

\begin{conjecture}
\label{conj:elastic}
In the lattice model with elastic collisions, the system consists of a single stable cloud if and only if $U_k > U_{N+1}$ for all $k \in [N]$. Moreover,
if the system is stable, there is a strong law of large numbers with limiting speed~$U_{N+1}$.
\end{conjecture}

As this model is not the focus of the present paper, we do not attempt to establish Conjecture~\ref{conj:elastic} here. Additionally, we anticipate that a similar algorithm to Algorithm~\ref{alg:partition},
based on the (somewhat simpler) algebra of cloud speeds in the elastic case, can be developed to obtain the full cloud decomposition for the elastic model.

This stability criterion in Conjecture~\ref{conj:elastic} coincides with ours in Corollary~\ref{cor:stable} when $N=1$ (but not for $N \geq 2$),
and even for $N=1$, the limiting speed is different. As an example, take
 $N=1$ and $a_1 = a$, $b_1 = b$, $a_2 = b_2 = 1$. Then the stable case of the exclusion model $(b>a)$ has
 asymptotic speed $\hv_2 = (b-a)/(1+b)$, while the stable case of the elastic model has asymptotic speed $U_2 = (b-a)/2$. 

\subsection{Scaling limits and heavy traffic}
\label{sec:scaling}

We indicate briefly, by an example, how the diffusion models of Section~\ref{sec:diffusions} arise as scaling limits of certain near-critical particle system models,
and that the limits are elastic even when the particle system is not: see~\cite[\S 3]{kpsAIHP}
for more detail. As an elementary starting point, suppose that $X^{(s)}_t$, $t\in \RP$ is a continuous-time, simple random walk on $\Z$ with jump rate~$1$ to the left and $1 + s^{-1/2} u$ to the right,
where $u \in \R$ and $s \in (0,\infty)$ is a scale parameter. Then $\IE ( X^{(s)}_{s(t+\delta)} - X^{(s)}_{st} ) = u \delta \sqrt{s}$,
$\lim_{s \to \infty} \IE ( ( X^{(s)}_{s(t+\delta)} - X^{(s)}_{st} )^2 / s) = 2\delta$, and standard functional central limit theorem arguments show that
$( s^{-1/2} X_{st} )_{t \in [0,1]}$ converges weakly, as $s \to \infty$, to $(x_t)_{t \in [0,1]}$, a diffusion with drift $u$ and volatility~$2$, i.e.,
satisfying the SDE~$\ud x_t = u \ud t + 2 \ud W_t$. 

More generally, consider the particle model with exclusion interaction from Section~\ref{sec:intro}, with $N+1$ particles with jump rates
$a_1 = \cdots = a_{N+1} = 1$ and $b_i = 1 + s^{-1/2} u_i$ for $i \in [N+1]$, for parameters $u_i \in \R$. Then, ignoring the interactions,
as $s \to \infty$ the particle motions converge to diffusions with drifts $u_i$ and volatility~$2$. In the presence of exclusion interactions, 
it can be shown that the particle system
in this case converges weakly, as $s \to \infty$, to the system~\eqref{eq:SDE-ordered-symmetric} with $\sigma_i^2 \equiv 2$; this is a special 
case of the results of~\cite[\S 3]{kpsAIHP}. Equivalently, in the Jackson network context, this is a \emph{heavy traffic} limit, as each queue is asymptotically critically loaded~\cite{reiman};
see e.g.~\cite{gz} for some of the vast literature on heavy-traffic limits for queueing networks and their connections to reflecting diffusions.

The apparent discrepancy between the stability criterion for the inelastic (exclusion) system in Corollary~\ref{cor:stable} with that for the elastic (diffusion) system
in Theorem~\ref{thm:diffusion} can now be resolved. For simplicity, take $N=2$, as in Example~\ref{ex:small}. Then the stability condition is that~\eqref{eq:3-particles-rho1} and~\eqref{eq:3-particles-rho2} both hold;
for $a_i=1$ and $b_i = 1 + s^{-1/2} u_i$, inequality~\eqref{eq:3-particles-rho1} takes the form $2 u_1 + u_2 > 2 u_2 + u_3 + o(1)$, as $s \to \infty$, which, in the limit,
gives $U_1 > U_3$ in the notation of Theorem~\ref{thm:diffusion}. Similarly~\eqref{eq:3-particles-rho2} translates in the limit to the condition $U_2 > U_3$. Thus
 the stability condition in Theorem~\ref{thm:diffusion} emerges via a particular limiting regime of our Corollary~\ref{cor:stable}. 
Moreover, the stable limiting speed~\eqref{eq:3-particles-speed} converges to the simple linear average $U_3$ in the $s \to \infty$ limit.

\section*{Acknowledgments}

The authors are grateful to Sunil Chhita for directing us to some relevant  literature, and to Aleksandar Mijatovi\'c for helpful discussions on diffusions with rank-based coefficients.
We also thank several anonymous referees whose constructive observations and suggestions on earlier versions of the manuscript
have contributed much to the present version.   
The work of MM and AW was supported by EPSRC grant EP/W00657X/1.

\end{document}